\documentclass[aop]{imsart}

\RequirePackage{amsthm,amsmath,amsfonts,amssymb}
\RequirePackage[numbers]{natbib}
\RequirePackage[colorlinks,citecolor=blue,urlcolor=blue]{hyperref}
\RequirePackage{graphicx}

\startlocaldefs
\usepackage[shortlabels]{enumitem}
\usepackage{tikz}
\usetikzlibrary{arrows.meta}
\usepackage{subfigure, float}
\usepackage{comment}
\usepackage{mathrsfs}
\usepackage{mathtools}
\mathtoolsset{showonlyrefs}

\newcommand{\commHL}[1]{#1} 
\newcommand{\commTJ}[1]{#1} 

\newcommand{\E}{\mathbb{E}}
\newcommand{\Var}{\mathop{\mathrm{Var}}}
\newcommand{\eps}{\epsilon}
\renewcommand{\P}{\mathbb{P}}
\newcommand{\TT}{\mathbb{T}}
\newcommand{\ZZ}{\mathbb{Z}}
\newcommand{\1}{\mathbf{1}}
\newcommand{\f}{\frac}

\newcommand{\ind}[1]{\mathbf{1}{\{ #1 \}}}

\newcommand{\norm}[1]{\lVert #1 \rVert}

\renewcommand{\root}{\mathbf{0}}

\newcommand{\BT}{\vec{\mathcal{T}}_{2d}}

\newcommand{\eqd}{\overset{d}=}
\newcommand{\Aa}{\mathcal{A}}
\newcommand{\Ss}{\mathscr{S}}

\newcommand{\Ff}{\mathscr{F}}

\newcommand{\aaa}{\mathfrak{a}}

\newcommand{\bigmid}{\;\big\vert\;}
\newcommand{\biggmid}{\ \bigg\vert\ }
\newcommand{\Biggmid}{\ \Bigg\vert\ }
\DeclarePairedDelimiter\abs{\lvert}{\rvert}
\DeclarePairedDelimiter\floor{\lfloor}{\rfloor}
\DeclarePairedDelimiter\ceil{\lceil}{\rceil}

\DeclareMathOperator{\Poi}{Poi}
\DeclareMathOperator{\Ber}{Ber}
\DeclareMathOperator{\Bin}{Bin}

\usepackage{theoremref}

\newcommand{\stprec}{\preceq}
\newcommand{\stsucc}{\succeq}

\newcommand{\lrprec}{\preceq_{\text{lr}}}

\theoremstyle{plain}

\newtheorem{theorem}{Theorem}[section]
\newtheorem{lemma}[theorem]{Lemma}
\newtheorem{proposition}[theorem]{Proposition}
\newtheorem{remark}[theorem]{Remark}
\theoremstyle{remark}

\endlocaldefs

\begin{document}

\begin{frontmatter}
\title{Particle density in diffusion-limited annihilating systems}
\runtitle{Particle density in DLAS}

\begin{aug}
\author[A]{\fnms{Tobias}~\snm{Johnson}\ead[label=e1]{Tobias.Johnson@csi.cuny.edu}},
\author[B]{\fnms{Matthew}~\snm{Junge}\ead[label=e2]{Matthew.Junge@baruch.cuny.edu}  },
\author[C]{\fnms{Hanbaek}~\snm{Lyu}\ead[label=e3]{hlyu@math.wisc.edu}},
\and 
\author[D]{\fnms{David}~\snm{Sivakoff}\ead[label=e4]{dsivakoff@stat.osu.edu}
}
\address[A]{Mathematics,
CUNY College of Staten Island\printead[presep={,\ }]{e1}}

\address[B]{Mathematics,
CUNY Baruch College\printead[presep={,\ }]{e2}}

\address[C]{Mathematics,
University of Wisconsin - Madison\printead[presep={,\ }]{e3}}

\address[D]{Mathematics and Statistics,
Ohio State University \printead[presep={,\ }]{e4}}

\end{aug}

\begin{abstract}
Place an $A$-particle at each site of a graph independently with probability $p$ and otherwise place a $B$-particle. $A$- and $B$-particles perform independent continuous time random walks at rates $\lambda_A$ and $\lambda_B$, respectively, and annihilate upon colliding with a particle of opposite type. 
	Bramson and Lebowitz studied the setting $\lambda_A = \lambda_B$ in the early 1990s. Despite recent progress, many basic questions remain unanswered when $\lambda_A \neq \lambda_B$. For the critical case $p=1/2$ on low-dimensional integer lattices, we give a lower bound on the expected number of particles at the origin that matches physicists' predictions. For the process with $\lambda_B=0$ on the integers and on the bidirected regular tree, we give sharp upper and lower bounds for the expected total occupation time of the root at and approaching criticality. 
	
\end{abstract}



\end{frontmatter}

	\section{Introduction}
	\label{sec:intro}

We consider two-type diffusion limited annihilating systems (DLAS) on integer lattices and directed regular trees. Initially every vertex has a particle that is independently of type~$A$ with probability $p$ and otherwise is of type~$B$. 
In continuous time, $A$-particles perform simple random walk at rate $\lambda_A$ and $B$-particles at rate $\lambda_B$. 
If two particles of opposite types collide, both are removed from the system. Particles annihilate pairwise and there is no limit to the number of like particles that can occupy a single site. See Figure \ref{fig:sim_Z} for a depiction of a simulation.

\begin{figure}[h]
	\begin{center}
		\includegraphics[width = .75 \textwidth]{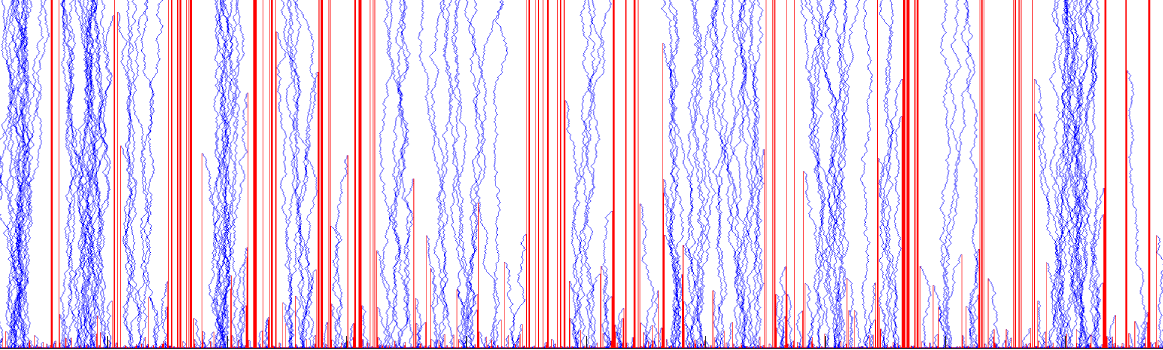}
	\end{center}
	\caption{Space-time representation of a DLAS with $p=1/2$ and $\lambda_B=0$ on the path with $2000$ vertices. Vertical lines are $B$-particles. Time runs from bottom to top.  
	}
	\label{fig:sim_Z}
\end{figure}

Physicists have been interested in this process as a model for irreversible reactions with mobile particles since the papers \cite{chem1} and \cite{chem2}.
The mean-field prediction for the model is that the density of both particle types should decay
at rate $t^{-1}$ if $p=1/2$, while the density of the less common particle type should decay
exponentially if $p\neq 1/2$. It was widely observed in the physics literature that while
these predictions were correct in high dimension, the model followed different asymptotics in low dimension.
But as Bramson and Lebowitz noted in \cite{BL2}, ``the answers given in that literature do not
always agree.'' 
They then rigorously determined the asymptotics of the model on $\ZZ^d$
for all values of $d$ when $\lambda_A=\lambda_B$ in \cite{BL4}. \commHL{Note that they consider slightly different initial conditions than our independent Bernoulli field. In their model, the number of $A$- and $B$-particles at each site are given by independent Poisson fields with intensities $r_A$ and $r_B$ for $A$- and $B$-particles, respectively.} Let
$\rho_t$ denote the expected number of $A$-particles at the origin at time~$t$.
Bramson and Lebowitz showed that when $r_A = r_B$ (analogous to $p=1/2$ in our model),
\begin{align}
\rho_t \asymp  \begin{cases}
 	t^{-d/4}, &d \leq 3 \\ 
 	t^{-1}, & d \geq 4
 \end{cases}\label{eq:rho}
\end{align}
using the notation $f(t)\asymp g(t)$ to denote that $f(t)/g(t)$ is bounded from above
and below by positive constants (which may depend on $d$ and $p$) for all $t$. 
On the other hand, when $p<1/2$,
\begin{align}\label{eq:rho2}
  e^{-c_d\psi_d(p,t)}\leq \rho_t \leq e^{-C_d\psi_d(p,t)},
\end{align}
\commHL{for all large $t$} 
where $c_d$ and $C_d$ are positive constants depending only on $d$ and
$$\psi_d(p,t) = \begin{cases}
 	\frac{(1-2p)^2}{1-p}\sqrt{t}, & d=1 \\ (1-2p)\frac{t}{\log t}, & d=2 \\ (1-2p)t, & d \geq 3.
 \end{cases}$$
This confirms that the model deviates from the mean-field behavior in dimensions $d\leq 3$
when the initial particle densities are equal and in dimensions $d\leq 2$ when the densities are unequal.

The closely related system \emph{annihilating random walk}  was already known not to exhibit mean-field behavior in low dimension. 
This process typically starts with one particle per site with particles performing independent random walks at rate $1$. Any collision results in mutual annihilation. 
Arratia proved in \cite{arratia1981limiting}
that $\beta_t$, the density of particles at the origin at time $t$, satisfies
\begin{align}\beta_t \asymp
 \begin{cases} t^{-1/2}, & d= 1 \\ 
t^{-1} \log t, & d =2 \\ t^{-1}, & d \geq 3\end{cases}.  \label{eq:beta} 
\end{align}
Arratia's work built on similar results from Bramson and Griffeath \cite{first} for \emph{coalescing random walk}, in which particles coalesce upon colliding (equivalently one particle is annihilated in each collision). 
\commHL{Coalescing random walk} has nice monotonicity properties and also enjoys a dual process known as the \emph{voter model}. \commHL{ Bramson and Griffeath \cite{first} used this dual process to obtain sharp asymptotic for particle density in coalescing random walk, and Arratia \cite{arratia1981limiting} used a coupling between coaleasing and annihilating random walk to show that the particle density in the annihilating system is asymptotically the half of the coalescing system. This leads to the exponents in \eqref{eq:beta}. There is no known tractable dual process for DLAS or coupling to well-known processes. }

The asymptotics given in \eqref{eq:rho} and \eqref{eq:rho2} were also conjectured to hold for \commHL{different} jump rates \cite{AB1,AB2,AB3}. 
However, a lack of symmetry makes these dynamics more difficult to analyze. 
For example, Bramson and Lebowitz analyzed DLAS with a coupling that first ignores the type of a particle and later reveals it. Such a coupling is valid only when $\lambda_A = \lambda_B$. {Another key result, which appears to hold only in the same-speed setting, is \cite[Lemma 2.1]{BL4}. This states that the average value of any convex function of the difference of the numbers of $A$- and $B$-particles in a region at time $t$ will be at least as large for DLAS compared to the process with no annihilation.}
A bound on $\rho_t$ when $\lambda_A\neq\lambda_B$
was proven in \cite{vicius_DLA} by Cabezas, Rolla, and Sidoravicius: 
for any choice of jump rates $\lambda_B\geq 0$ and $\lambda_A >0$, it holds that $\rho_t\geq C/t$ on a large class of transitive, unimodular graphs \commHL{(including $\mathbb{Z}^{d}$)} when the particle densities are in balance. \commHL{In particular, this is in line with \eqref{eq:rho} for $d\ge 4$.} The same authors in \cite{Cabezas2014} proved that when $\lambda_B=0$ and $A$-particles
move with drift, the number of visits to the origin in time~$t$ is of order $\sqrt{t}$. \commHL{It turns out, however, that the asymptotic behavior for $p<1/2$ is different when $\lambda_B=0$ from when $\lambda_A=\lambda_B>0$. Damron, Lyu and Sivakoff show that when $\lambda_B=0$ and $A$-particles move as discrete-time symmetric random walks, then $\rho_t  = \exp(-O(t^{d/(d+2)}))$, which is distinct from~\eqref{eq:rho2} for all dimensions $d\ge 1$ \cite{DLS}. This cast some doubt on whether the same asymptotics as in~\eqref{eq:rho} and~\eqref{eq:rho2} should always be  expected when $\lambda_B<\lambda_A$.}

Some other related results include: Cristali, Jiang, Junge, Kassem, Sivakoff, and York considered a discretized version of DLAS on finite graphs and studied the time to extinguish all particles \cite{dlas_star}. Ahlberg, Griffiths, and Janson studied the critical behavior of an two-type annihilating system
of branching random walks \cite{dlas_brw}. Dauvergne and Sly recently studied a variant of DLAS in which $A$- and $B$-particles move at different rates but, rather than collisions resulting in mutual annihilation, $B$-particles are converted to $A$-particles upon contact \cite{dauvergne2021spread}. 
Bahl, Barnet, Junge, and Johnson proved in \cite{bahl2021diffusion} that the occupation time of a subset of vertices by $A$-particles in DLAS is monotonic as the initial configuration is augmented in the increasing convex order, a nonstandard stochastic order that rewards volatility. 

Cabezas, Rolla, and Sidoravicius further proved that DLAS undergoes a phase transition from infinite visits by $A$-particles to the root (recurrence) when $p \geq 1/2$ to only finitely many visits (transience) for $p<1/2$ \cite{vicius_DLA}. 
This built on previous recurrence/transience results by the same authors \cite{Cabezas2014} for the special case $\lambda_B=0$, which they named the \emph{particle-hole model}. An \commHL{Abelian} property ensures their results also hold in discrete time. 
Damron, Gravner, Junge, Lyu, and Sivakoff produced similar results for the case $\lambda_B=0$ with $A$-particles performing discrete time random walk and also derived some quantitative estimates on the expected number of particles to visit the origin in \cite{parking}. 
Inspired by recent results in parking on random graphs \cite{uniform_parking, tree}, the authors named this setting the \emph{parking process}. 

As we were writing up our results, Przykucki, Roberts, and Scott released a paper concerning the parking process on the integers \cite{parking_on_integers}. For the case $p=1/2$, they proved lower and upper bounds on the expected occupation time of the root, $\E_p V_t = \int_0^t \rho_s ds,$ that matched the conjectured behavior up to a sublogarithmic factor:
\begin{align}t^{3/4}(\log t)^{-1/4} \lesssim \E_{1/2} V_t \lesssim t^{3/4}.\label{eq:integers1}	
\end{align}
\commHL{Here, the occupation time $V_{t}$ is the integral of the number of $A$-particles at the origin at time $s$ from $s=0$ to $t$, so it accounts for multiple visits of the same particle to the origin up to time $t$.} The notation $f(t)\lesssim g(t)$ means $f(t) = O(g(t))$.
Addressing a conjecture from \cite{parking} concerning the rate of growth of $\E_p V_\infty$ as $p$ approaches $1/2$, they further proved that as $p \uparrow 1/2$ 
\begin{align}
\E_p V_\infty \lesssim (1-2p)^{-6}.\label{eq:integers2}
\end{align}
Using different techniques, we prove a stronger version of this estimate.
We also provide several other estimates on $\rho_t$, which we summarize below and then state more precisely.

For $d\leq 3$, we give a lower bound of $\rho_t\geq (t\log t)^{-d/4}$ for the $p=1/2$ case
(see Theorem~\ref{thm:LB}).
This is consistent with the behavior with equal jump rates in \eqref{eq:rho} and 
with $\lambda_B=0$ in \eqref{eq:integers1} for the case $\lambda_B=0$ in discrete time on $\mathbb Z$. Our work and \cite{parking_on_integers} are the first confirmations of deviation from mean-field behavior with nonequal jump rates. 
For $d=1$ and $\lambda_B=0$, 
we also give an upper bound on the total occupation time for a site that essentially
confirms that the asymptotics of \eqref{eq:rho} hold in this case (Theorem~\ref{thm:UB}). This agrees with the upper bound in \cite{parking_on_integers} except ours is proven in continuous time. Addressing \eqref{eq:integers2}, we prove that, up to a logarithmic factor, $\E_p V_\infty$ grows like $(1-2p)^{-3}$ (Theorems~\ref{thm:EV_LB} and \ref{thm:EV_UB}). 

Our final results investigate the high-dimensional behavior. Bramson and Lebowitz showed that
DLAS has 
the mean-field density decay of $t^{-1}$ on $\ZZ^d$ for $d\geq 4$ in the case of equal jump rates.
We consider the model with $\lambda_B=0$ on a directed regular tree.
We give upper and lower bounds of order $\log t$ on the cumulative occupation time for the 
$p=1/2$ case, and we give upper and lower bounds as $p\uparrow 1/2$ that agree up to constants
(Theorems~\ref{thm:tree} and
\ref{thm:tree.exponent}). To the best of our knowledge, this is the only instance in which mean-field
behavior has been proven to occur with nonequal jump rates on an infinite graph.

\subsection{Statement of results}
\label{subsection:results}
We consider the two-type DLAS on a given rooted graph where each vertex initially
contains exactly one particle, which has type~$A$ with probability $p$,
with jump rates for the two particle types given by $\lambda_A$ and $\lambda_B$. 
Without loss of generality, we can take one of the jump rates to be $1$.
We take $\lambda_A=1$ in all results except for \thref{thm:EV_LB}, where doing so would result
in a needless loss of generality.

Let $N_t$ be the number of $A$-particles at the root at time~$t$.
Let $\rho_t=\E N_t$, which we refer to as the density of $A$-particles
at time~$t$. Finally, let $V_t=\int_0^t N_s\,ds$, the aggregate time spent by $A$-particles at the root
up to time~$t$. 
 All of our results concern the asymptotic behavior of $\rho_t$ and $\E V_t$ on
transitive graphs, rendering the choice of root irrelevant.

Our first result is a general lower bound on density with particle types in balance, confirming the lower bounds (up to a logarithmic factor) of the Bramson--Lebowitz asymptotics \eqref{eq:rho} in low dimension. 

\begin{theorem} \thlabel{thm:LB} 
 Let $\lambda_A = 1$ and $0\leq \lambda_B\leq 1$. On $\ZZ^d$ with $d\leq 3$ and $p=1/2$ there exists a constant $C>0$ that does not depend on $\lambda_B$ such that  $$\rho_t \geq C ( t \log t)^{-d /4}$$
 for all sufficiently large $t$
 \end{theorem}


Next, we provide an upper bound on $\E V_t$ in dimension~$1$ with $\lambda_B=0$.
Combined with Theorem~\ref{thm:LB} and the fact that $\E V_t = \int_0^t \rho_s ds$, 
it strongly supports the conjecture that $\rho_t \asymp t^{-1/4}$.

\begin{theorem} \thlabel{thm:UB}
Let $\lambda_A =1$ and $\lambda_B=0$. On $\ZZ$ with $p=1/2$ there exists $C>0$ such that $$\E V_t \leq C t^{3/4}$$
for all sufficiently large $t$.
\end{theorem}


{Next, we consider the critical exponent of $\mathbb E_p V_{\infty}$ as $p \uparrow 1/2$.

\begin{theorem} \thlabel{thm:EV_LB}
 Let  $0 < \lambda_A \leq  1$ and $0 \leq \lambda_{B}\leq 1$. 	On $\ZZ^d$ for $d \leq 3$ there exists a constant $C>0$ such that 
	$$\E_p V_\infty \geq C \frac{\left( 1  - 2p\right )^{ -(4/d) + 1}}{-\log \left( {1-2p}\right)}$$
	for all $1/4< p < 1/2$.
\end{theorem}

\begin{theorem}\thlabel{thm:EV_UB}
Let $\lambda_{A}=1$ and $\lambda_B=0$. On $\mathbb Z$, there exists $C>0$ such that
$$\E_p V_\infty \leq C \left( 1 - 2p\right)^{-3}$$
for all $p < 1/2$.
\end{theorem}

For $d=1$ with $\lambda_B=0$, \thref{thm:EV_LB,thm:EV_UB} show that
\begin{align*}
 \frac{(1-2p)^{-3}}{-\log(1-2p)} \lesssim \E_pV_\infty\lesssim (1-2p)^{-3},
\end{align*}
determining the critical exponent up to logarithmic terms.
%
While our upper bound on $\E_p V_\infty$ with $p<1/2$, $d=1$, and $\lambda_B=0$ is sharp
(at least up to logarithmic factors), we mention that for $d\geq2$ and $\lambda_B<\lambda_A$
it remains open to show the much weaker statement that $\E_p V_\infty<\infty$ for all $p<1/2$.




Our final results concern the mean-field behavior $\rho_t\asymp t^{-1}$. \commHL{This was}
first proven by Bramson and Lebowitz on lattices of dimension 4 and higher
in the case of equal \commHL{particle densities and} jump rates. \commHL{While it is believed to hold for all jump rates, it has not been proven to occur with unequal jump
rates on any graph.} We consider DLAS with $\lambda_B=0$ on $\BT$, 
the bidirected $2d$-regular tree, which is the $2d$-regular tree where each vertex has $d$ edges directed away from it and $d$ edges directed towards it.
We prove that $\E V_t$ diverges logarithmically, which strongly suggests the mean-field
density decay.
The rationale for working on $\BT$ is that the $A$-particles approaching the root from different branches
of the tree evolve independently, but analysis remains difficult even with this advantage.

\begin{theorem} \thlabel{thm:tree}
Let $\lambda_A=1$ and $\lambda_B=0$. For some positive absolute constants $c$ and $C$,
  it holds for all $d\geq 2$ on $\BT$ with $p=1/2$ that
	$$c\log t\leq \E V_t\leq C \log t$$
	for all large $t$. 
\end{theorem}

Finally, as in Theorems~\ref{thm:EV_LB} and \ref{thm:EV_UB}, we investigate how quickly
$\E_p V_\infty$ diverges as $p\uparrow 1/2$.
\begin{theorem}\thlabel{thm:tree.exponent}
	Let $\lambda_A=1$ and $\lambda_B=0$. For some positive absolute constants $c$, $C$, and $\eta$, it holds 
  on $\BT$ for all $\frac{1}{2}-\eta<p<\frac{1}{2}$ and $d\geq 2$ that
  \begin{align*}
    c\log\left(1-2p\right)^{-1}\leq \E_p V_{\infty}\leq C\log\left(1-2p\right)^{-1}.
  \end{align*}
\end{theorem}

\subsection{Definitions and notation} \label{sec:notation}

DLAS on a graph with vertex set $\mathcal{V}$
is a continuous-time Markov process $\zeta_t = (\zeta_t(v))_{v \in \mathcal{V}}$
on state space $\ZZ^\mathcal{V}$.
The quantity $\abs{\zeta_t(v)}$ denotes the number
of particles at site~$v$ at time~$t$. The sign of $\zeta_t(v)$ is positive if these
particles are of type~$A$ and negative if they are of type~$B$.
The dynamics of the process are as described earlier:
particles of types $A$ and $B$ jump at rates $\lambda_A$ and $\lambda_B$,
respectively;
a particle at vertex~$u$
takes its next step to $v$ with probability $K(u,v)$, where $K$ is a given
random walk kernel; and when a particle jumps onto a site with a particle of the opposite type,
both particles are instantly annihilated. The infinitesimal generator corresponding to this description
is given explicitly in \cite[Section~2]{vicius_DLA}. A graphical construction
proving that such a Markov process exists is sketched in \cite{BL4} and given in \cite{vicius_DLA}.

If we do not mention
the random walk kernel $K$ specifically, then we take it to be the nearest-neighbor simple random walk kernel
on the given graph.
Our default initial conditions consist of one
particle per site, each of which independently is an $A$-particle with probability~$p$
or a $B$-particle with probability~$1-p$. We will also frequently consider these initial conditions
restricted to a subset of the graph, meaning that rest of the graph is initially devoid of particles.

In the previous section, we defined the quantities $N_t$, $\rho_t$, and $V_t$. 
We will also use the notation $D(H)$ to denote the \emph{discrepancy} between
$A$-particles and $B$-particles on a subgraph $H$, defined as the number of $A$-particles
minus the number of $B$-particles initially placed in a subgraph $H$ in a given instance of DLAS. We let $\mathbb T_{r}^d$  denote the $d$-dimensional torus of radius $r$, which has vertex set $(-r,r]^d\cap \ZZ^d$ and nearest neighbor edges with the canonical identifications of opposite sides.
We denote by $\root$ the origin of the lattices and torus, and also the root of the bidirectional tree.

\subsection{Overview of proofs}
\label{sec:proof.overview}

Our results rely on a variety of couplings that allow us to make comparisons to modified versions
of the systens. 
We sketch the main ideas below.

\subsubsection*{\thref{thm:LB,thm:EV_LB}, lower bounds for $d=1,2,3$}
The idea behind \thref{thm:LB} is to compare $\rho_t$ to the particle density of $\mathbb T_{2r}^d$ with $r = C \sqrt{t\log t}$, which we denote as $\bar \rho_t$. 
The width of the torus is chosen so that number of particles at the origin up to time~$t$
is unlikely to be affected by the particles beyond distance $r$ in the processes on $\mathbb Z^d$ and on $\mathbb T_{2r}^d$ (see \thref{lem:torus}). So it is enough to estimate $\bar \rho_t$.

As the number of vertices in $\mathbb T_{2r}^d$ is on the order of $r^d$, standard \commHL{central limit theorem} estimates give that with positive probablity, there are $r^{d/2}$ more $A$- than $B$-particles in the initial configuration. Since the torus is a finite graph, these surplus $A$-particles are never annihilated.
Translation invariance ensures that   
\begin{align*}
  \bar\rho_t\gtrsim \frac{r^{d/2}}{\abs{\TT_{2r}^d}} \gtrsim (t \log t)^{-d /4}.
\end{align*}

The critical exponent bound given in Theorem~\ref{thm:EV_LB} uses the same idea, but optimizes the size of the torus as a function of $p$ and uses more precise estimates for the $A$-particle surplus in the initial configuration. Similar ideas are used when making estimates on the \emph{correlation length} in first passage percolation \cite{fpp}. Note that Bramson and Lebowitz \cite{BL4} also made use of fluctuations in the initial configuration when studying the symmetric speeds case. See the heuristic at \cite[p.\ 4]{BL4}.

\subsubsection*{\thref{thm:UB,thm:EV_UB}, upper bounds for $d=1$}
The starting point is to consider a sequential version of DLAS in which all $B$-particles are initially present and $A$-particles are released one 
at a time according to an arbitrary prescribed ordering. The first $A$-particle travels until it annihilates with a $B$-particle or time~$t$ elapses.
Then, in this new environment, the next $A$-particle is released and does the same, and so on.
For $\lambda_B =0$, \thref{lem:sequential.process} establishes that
the occupation time of the root is stochastically larger in this variant than in the original process. Moreover, the subadditivity result in \thref{lem:polarized_DLAS} reduces the problem
to studying the one-sided version of DLAS with particles only at the positive integers.

For \thref{thm:UB}, we run the sequential process on the half-line,
releasing $A$-particles in order of their distance from the origin.
We show that an $A$-particle starting at $k$
has probability $O(k^{-1/2})$ of visiting the origin in time~$t$ (\thref{lem:seq}). 
Summing over all $k$ from $1$ to $\sqrt{t}$ and using random walk concentration bounds
to bound the contributions of particles starting beyond distance $\sqrt{t}$,
we obtain a bound of $O(t^{1/4})$ on the number of distinct particles
visiting the origin in the half-line process. 
If a particle visits the origin,
we expect it to spend at most time $O(t^{1/2})$ up to time $t$ there, by basic properties of random walk,
yielding the $O(t^{3/4})$ bound for the one-sided sequential process.



The sequential release of particles is essential to the proof of \thref{lem:seq}, which gives a bound
on the probability of $A$-particles visiting the origin. The main idea
is that for an $A$-particle starting at $k$ to reach $\root$, the $A$-particles in $[1,k-1]$
must have already annihilated all $B$-particles that were initially  there.
The typical surplus of the $A$-particles against the $B$-particles in $[1,k-1]$ is $O(k^{1/2})$,
and the probability of the $A$-particle at $k$ reaching $\root$ is maximized if all the surplus
$A$-particles annihilate $B$-particles to the right of position $k$. 
We then estimate the chance of the $A$-particle at $k$ reaching $\root$ using a refined ``gambler's ruin'' estimate (\thref{lem:gamblers.ruin}), in which we bound the chance of the particle hitting $\root$ by time $t$ and
before visiting the first remaining $B$-particle located approximately $k^{1/2}$ steps to its right.

{To prove \thref{thm:EV_UB}, as in the proof of \thref{thm:UB} we use 
a gambler's ruin approach to bound the time spent at $\root$ by an $A$-particle starting from
position~$k$ in terms of the number of surplus $A$-particles against $B$-particles
on $[0,k]$. When we sum this bound over all $k$, we obtain a bound on $\E V_t$ in terms
of the expected area underneath the positive excursions of a $p$-biased random walk
for $p<1/2$, which we then compute.}

It seems challenging to us to generalize our approach to higher dimensional lattices, because there is no longer a simple  way to control  the probability an $A$-particle at distance $k$ reaches the origin. One would have to understand the spatial correlations between unvisited $B$-particles as the process evolves. These correlations may be significant (see \cite[Figure 2]{parking}.) In the $\lambda_B>0$ case, it also appears difficult to make analogous estimates even in $d=1$, since
the coupling of the sequential version of DLAS with the usual one depends on $B$-particles remaining still.

\subsubsection*{\thref{thm:tree,thm:tree.exponent}, results on bidirected tree $\BT$}
These results are proven rather differently from the other estimates.
To avoid some technical issues in this explanation, consider DLAS in discrete time.
Let $W_n$ denote the number of visits to the root by $A$-particles in $n$ steps of this process and
let $v_1,\ldots,v_d$ be the vertices whose out-edges lead to the root
of $\BT$.
We can express $W_{n+1}$ in terms of the number of visits to  $v_1,\ldots,v_d$
in time~$n$. From the self-similarity of the tree, each of these quantities is an independent
copy of $W_n$. We thus obtain a distributional equation \eqref{eq:W_n.Aa} giving the law of $W_{n+1}$ in terms
of the law of $W_n$. Analysis of this equation shows that the growth of the mean of $W_n$ depends on
its concentration (see \thref{lem:A.growth}). 

We then prove concentration and anticoncentration bounds for $W_n$.
The anticoncentration bound (\thref{lem:lower.nonconcentration}) proves the
lower bounds in \thref{thm:tree,thm:tree.exponent}. 
\commHL{Note that Cabezas, Rolla, and Sidoravicius show a general result 
\cite[Theorem~4]{vicius_DLA} that
$\rho_t=\Omega(1/t)$ when $p=1/2$ on generously transitive graphs with reflectable jump distributions. The lower bound in \thref{thm:tree} of matching order, however, does not follow from this general result, since the jump distribution on directed bidirectional tree is not reflectable.} The upper bounds on $\E W_n$ are a consequence of the concentration
bound on $W_n$, which we prove with the technique of size-bias couplings.
The idea of this technique is that stochastic
inequalities between a random variable and its size-bias transform lead to
concentration inequalities for the random variable. The problem then turns to 
computing the size-bias transform of $W_n$ and showing that it does not differ from $W_n$ by
too much. To size-bias a sum of independent terms, one chooses a single term at random to bias, leaving
the others unaffected (hence the title of the survey paper \cite{AGK}). Thus, to size-bias $W_n$,
we bias the number of root visits coming from one of its children. Recursively, this creates a ray
on which the process is altered. The result is something like placing an extra $A$-particle at every vertex 
along this ray, and concentration then follows from showing that this adds only $O(1)$ visits to the root.
The actual details are more complicated; see Lemma~\ref{lem:spine.induction} and Proposition~\ref{prop:sb.coupling} and the discussion thereafter.

These size-biasing techniques seem novel to us in the context of particle systems.
To put them in context, size-biasing has a long history in Stein's method for distributional
approximation (see \cite[Sections~3.4 and 4.3]{Ross}). More recently, size-biasing methods have
been developed for proving concentration \cite{GG,AB,CGJ}. On a different track,
a technique of creating a spine with altered behavior to bias a statistic of a Galton--Watson tree was used
in \cite{LPP} to prove the Kesten--Stigum Theorem 
(other good references are \cite[Chapter~12]{lyons2016probability} and \cite[Section~7]{AGK}). 
This technique was later used to prove
many results on branching processes; for example, see \cite{HH} and \cite[Chapters~4 and 5]{Shi}.
These two uses of size-biasing, Stein's method and spine techniques, finally met
in \cite{PekozRollin} where Stein's method together with the construction from \cite{LPP}
is used to prove a quantitative version of Yaglom's theorem on critical Galton--Watson trees.
Of all uses of size-biasing, the most relevant to this paper is used in \cite{+recursion1} to analyze
the Derrida--Retaux model from statistical physics. This model is essentially
DLAS with $\lambda_B=0$ but on a directed rather than bidirected tree; see
\cite{DerridaRetaux,+recursion2,+recursion3}.

\subsection{Organization}

Section~\ref{sec:lemmas} contains statements of some important lemmas 
as well as descriptions of variants of DLAS that we relate to the original process. Section~\ref{sec:lemmas.proofs} contains the proofs of these lemmas. We prove our main results for DLAS on lattices (Theorems \ref{thm:LB}, \ref{thm:UB}, \ref{thm:EV_LB}, and \ref{thm:EV_UB}) in Section \ref{sec:lattice}. In Section~\ref{sec:tree}, we prove our main results on DLAS on bidirected trees (Theorems~\ref{thm:tree} and \ref{thm:tree.exponent}). The appendix contains some useful random walk estimates.

\section{Key lemmas}
\label{sec:lemmas}

In this section, we present a toolkit of lemmas for DLAS, whose proofs are given in Section~\ref{sec:lemmas.proofs}. A reader more interested in how the lemmas are applied may safely read the statements and skip ahead to Section~\ref{sec:lattice}.
We start with a quick overview of the lemmas and where they are used.

	\begin{itemize}
		\item \thref{lem:monotonicity} asserts that removing $A$-particles results in fewer visits to the root by $A$-particles. We use this only in one place in Section~\ref{sec:tree}.
		\item \thref{lem:finite.range.truncation} shows that $\rho_t$ does not change much when we remove
      all particles beyond distance $\sqrt{C t \log t}$ from the root of the torus or lattice or beyond distance $ct$ from the root of the bidirected tree. We use the lemma in the proofs of \thref{lem:torus},
      \thref{thm:tree}, and \thref{thm:tree.exponent}.		
		\item In \thref{lem:torus} we relate DLAS on $\mathbb Z^d$ to DLAS on a torus. We use this to prove
      the lower bounds on the lattice, \thref{thm:LB,thm:EV_LB}.
		
		\item \thref{lem:sequential.process} shows that when $\lambda_B=0$,
      the occupation time of $\root$ in DLAS is stochastically larger if we release $A$-particles
      one at a time, letting each one run for a fixed time~$t$ or until annihilation before running
      the next. This result is crucial for proving \thref{thm:UB,thm:EV_UB},
      our upper bounds for one-dimensional DLAS.
		\item \thref{lem:polarized_DLAS} is a subadditivity result saying that when $\lambda_B=0$, 
      dividing the particles
      of DLAS into two sets and running them as two independent systems makes the combined number
      of visits to the root stochastically larger.
      We need this for Theorems~\ref{thm:UB} and \ref{thm:EV_UB}.
	\end{itemize}

Now we state the lemmas precisely. {Several statements are for DLAS on transitive unimodular graphs i.e., graphs with a transitive unimodular automorphism group. This includes the graphs $\mathbb Z^d, \mathbb T_r^d, \vec{\mathcal{T}}_{2d}$ to which our theorems apply. We work at this level of generality when possible, since \cite[Section 6]{vicius_DLA} proved that the graphical construction of DLAS is well-defined in this setting. We describe the construction in Section~\ref{sec:tracers} and use it throughout our arguments.} 



\begin{lemma}\thlabel{lem:monotonicity}
  Let $\zeta$ and $\zeta'$ be two instances of DLAS on a transitive unimodular graph. 
  Let $V_t$ and $V'_t$ be the occupation times of the root by $A$-particles up to time $t$.
  If $\zeta_0(v)\leq\zeta'_0(v)$ a.s.\ for all $v$, then $E V_t \leq E V'_t$.
\end{lemma}

\begin{lemma}\thlabel{lem:finite.range.truncation}
  Let $\rho_s^{(r)}$ be the density of $A$-particles at the root in DLAS 
  with $0 \leq \lambda_B\leq 1$ and 
   \commHL{$\lambda_{A}=1$}
  at time~$s$
  with particles beyond graph distance $r$ from the root removed
  from the initial configuration in $\BT$ or outside of $[-r,r]^d$ removed from the initial configurations in  $\mathbb T_{2r}^d$ and $\mathbb Z^d$. 
  
  For $\BT$ with $d \geq 2$ and $r = \lceil c t \rceil$, it holds for some absolute constant $c$ and all $s\leq t$ that
  \begin{align*}
    \abs[\big]{\rho_s-{\rho}^{(r)}_s}\leq e^{-t}.
  \end{align*}
For the torus $\mathbb T_{2r}^d$ and for $\mathbb Z^d$,
  \begin{align*}
    \abs[\big]{\rho_t-{\rho}^{(r)}_t}\leq C_1 t^{d} e^{-C_0 r^2/t}
  \end{align*}
for constants $C_0,C_1>0$ and all $r,t$ satisfying $1\leq\sqrt t\leq r\leq 2t$.
\end{lemma}


\begin{lemma} \thlabel{lem:torus}
  Let $\rho_t$ be the density of $A$-particles for DLAS on $\ZZ^d$ with $0 \leq \lambda_B\leq 1$ and 
  \commHL{$\lambda_{A}=1$.}
  Let $\bar \rho_t$ be the density of $A$-particles for DLAS on
  the torus $\mathbb T_{2r}^d$ at time $t$ with $\sqrt t \leq r \leq 2t$. There exist $c,C>0$ that do not depend on $t$ and $r$ such that 
	$$|\rho_t - \bar \rho_t| \le C t^{d} e^{-cr^2/ t} $$
	for all $t \geq 1$. 
\end{lemma}

In Sections~\ref{sec:UB} and \ref{sec:EV_UB},
we will consider a variant of DLAS that we call the \emph{sequential process run for time~$t$}. 
Let $\lambda_B=0$ and fix a time $t\leq\infty$. Assume that DLAS on a given graph
has standard initial conditions restricted to a finite subset $H$ of the graph.
Place any ordering on the vertices of $H$, and let it induce an ordering on the $A$-particles according
to their initial positions.
With all other particles holding still,
let the first $A$-particle carry out its random walk up to time~$t$ or until
it hits a $B$-particle, in which case it and the $B$-particle are annihilated as usual.
Then, in this new landscape of $B$-particles, run the second $A$-particle in the same way, and continue
until all $A$-particles have gone.
Define the \textit{occupation time} of the root for this process to be
the sum of the times spent at the root by each $A$-particle.

We will show that the sequential process dominates the usual DLAS, in the sense
 that the occupation time of the root is stochastically larger in the sequential process
 than in the usual DLAS. For random variables $X$ and $Y$, we use the notation $X\preceq Y$ to denote that $X$ is \textit{stochastically dominated} by $Y$
 in the standard sense that $\P(X\leq t)\geq\P(Y\leq t)$ for all $t$, or equivalently
 that there exists a coupling of $X$ and $Y$ so that $X\leq Y$ a.s.

\begin{lemma}\thlabel{lem:sequential.process}
  Consider DLAS on a transitive unimodular graph $G$. 
  Let $H$ be a subset of the vertices of $G$ \commHL{and assume $\lambda_{B}=0$. }
  Let $V_t$ be  the total occupation time of the root
  by $A$-particles up to time $t\leq\infty$ in a DLAS with all $B$-particles present and only the $A$-particles initially in $H$ present. 
  For a given ordering of vertices in $H$,
  let $V'_t$ be the total occupation time of the root
  in the sequential process run for time~$t$. Then $V_t\preceq V'_t$.
\end{lemma}

Next, imagine dividing the particles in DLAS into two classes and running each
as a separate instance of DLAS. Intuitively, the combined particle density in the separated processes
should dominate the density in the origin DLAS, since more annihilations will occur in the original
process. 
We prove this for the $\lambda_B=0$ case.

\begin{lemma}\thlabel{lem:polarized_DLAS}
Consider DLAS on a transitive unimodular graph with given initial conditions with no more than one particle per site,
and with every particle labeled as positive or negative.
Then consider two independent DLAS, one with only the positive particles present and one with only the negative particles present. Let $V^+_t$ and $V^-_t$ be the total occupation times of the root \commHL{up to time $t$} in each of these processes. For all $0<t\leq\infty$, it holds that
	$$V_t \preceq V^+_t + V^-_t.$$
\end{lemma}

\section{Process construction and proofs of key lemmas}
\label{sec:lemmas.proofs}

{The goal of this section is to establish the lemmas from Section~\ref{sec:lemmas}. In Section~\ref{sec:tracers}, we describe the graphical construction of DLAS and a tracer system to track differences between different DLAS from \cite{vicius_DLA}. These tools help prove Lemmas \ref{lem:monotonicity}, \ref{lem:finite.range.truncation}, and \ref{lem:torus}. In Sections~\ref{sec:sequential} and \ref{subsection:polarized_construction}, we introduce two modified tracer constructions that we use to prove \thref{lem:sequential.process} and \thref{lem:polarized_DLAS}, respectively.}

{We note that \cite{vicius_DLA, parking_on_integers, Cabezas2014, rivera2019dispersion, bahl2021diffusion} developed couplings for DLAS.  As mentioned, \cite{vicius_DLA} introduced the tracer construction. A modified tracer system was later used in \cite{bahl2021diffusion}. A variant of DLAS in which $A$-particles can selectively ignore certain collisions
with $B$-particles is considered in \cite{parking_on_integers}. The paper also uses
a construction of DLAS with $\lambda_B=0$ where each site has a stack of instructions.
This perspective is a special case of a more general \commHL{Abelian} property for the equivalent particle-hole model mentioned earlier  \cite[Lemma 1]{Cabezas2014}.  A path concatenation scheme that uses Young tableaux is used in \cite{rivera2019dispersion} in a version of internal diffusion-limited aggregation, which can be thought of as a DLAS in which $A$-particles are released sequentially from a single source. Speaking to the subtlety of DLAS, all of these couplings are different from each other as well as from the modified tracer systems we introduce in Sections~\ref{sec:sequential} and \ref{subsection:polarized_construction}.
}

\subsection{Tracers and the proofs of Lemmas~\ref{lem:monotonicity}--\ref{lem:torus}}
\label{sec:tracers}
The essential idea in the proofs of \thref{lem:finite.range.truncation,lem:torus}
is that when considering the occupation time at a site, particles that begin far away from
the site have little effect on the occupation time and can be ignored.
To make this precise, we use a construction from \cite{vicius_DLA} in which
\emph{tracer particles} track the differences between two coupled DLAS with different initial
conditions.

First, we describe the graphical construction for DLAS from \cite{vicius_DLA}.
Fix $x \in G$ and assign to the $j$th particle counted by $\zeta_0(x)$ a discrete simple random walk path $W^{x,j} = (W^{x,j}_k)_{k \in \mathbb N \cup \{0\}}$ which it follows according to a rate $\lambda_A$ or $\lambda_B$ Poisson point process depending on whether $j$ is positive ($\lambda_A$) or negative ($\lambda_B$). This forms a continuous-time path $S^{x,j}$ called the \emph{putative trajectory} of the particle.
 We further assign to the $j$th particle counted by $\zeta_0(x)$ a \emph{braveness} $h^{x,j} \in [0,1]$ chosen uniformly at random.
Assume independence of $\{S^{x,j}, h^{x,j} \colon  x\in G, j\in \mathbb Z\setminus\{0\}\}$.
Particles follow their putative trajectories $S^{x,j}$. When one or more particles of opposite type occupy a site, the bravest $A$- and $B$-particles pairwise mutually annihilate
until there are no remaining pairs of opposite type particles at the site. 
It is shown in \cite[Section 6]{vicius_DLA} that for a transitive unimodular graph (such as $\BT$, $\mathbb T_{r}^d$, or $\mathbb Z^d$), this graphical model is well defined, and its particle counts at time~$t$
form a Markov process with the correct transition rates for DLAS.

The graphical construction allows us to couple DLAS with different initial configurations. 
Let $\zeta'$ be a DLAS defined using the same graphical construction as $\zeta$ but with initial conditions $\zeta_0'$.
Since we will only need to add or remove at most one particle from any given site, we will assume that $|\zeta_0'(x) - \zeta_0(x)| \leq 1$ for all $x \in G$. 
Let $\mathcal A^+ = \{ x  \colon \zeta_0'(x) - \zeta_0(x) =1\}$  and $\mathcal A^-= \{ x \colon \zeta_0'(x) - \zeta_0(x) = -1\}$. Set $\mathcal A = \mathcal A^+ \cup \mathcal A^-$.

We use a method of \emph{tracer particles} given in \cite{vicius_DLA} to track the difference
between $\zeta$ and $\zeta'$ as they evolve.
We now summarize their construction, described in more detail in \cite[Section 3.1]{vicius_DLA}.
(Though the construction is given within the section of \cite{vicius_DLA} dealing
with the case $\lambda_A=\lambda_B=1$, it works for general $\lambda_A$ and $\lambda_B$.)
To describe how tracers work, we follow \cite{vicius_DLA} and start with the case of a single
$A$-particle added at location $x$, i.e., $\zeta_0'(x)=\zeta_0(x)+1$ with $\zeta_0(x)\geq 0$.
We let $X^x=(X^x_t)_{t\geq 0}$ denote the path of the tracer particle we define now.
Initially the tracer follows this $A$-particle's path in $\zeta'$.
Now, suppose this $A$-particle is annihilated in $\zeta'$. Following the annihilation,
there is either a $B$-particle present
in $\zeta$ but not $\zeta'$ (if no other $A$-particle was present at the annihilation site),
or there remains an extra $A$-particle present in $\zeta'$ but not $\zeta$ (if another $A$-particle
was present at the annihilation site). The tracer then follows the extra $A$-particle
in $\zeta'$ or the extra $B$-particle in $\zeta$. When this particle is annihilated,
there will continue to be either an extra $A$-particle in $\zeta'$ or an extra $B$-particle in $\zeta$,
and the tracer transfers itself to this extra particle.
This process continues, and 
$\zeta_t'-\zeta_t=\delta_{X^x_t}$ for all $t$. If $\lambda_A=\lambda_B=1$, then $X^x_t$ is a random
walk in continuous time. If we only know that $\lambda_A\leq 1$ and $\lambda_B\leq 1$, then $X^x_t$
can be coupled with a rate~$1$ random walk so that in time~$t$ it visits
a subset of the sites visited by a rate-1 random walk. Though we described this construction
for the case where initially an extra $A$-particle is initially present in $\zeta'$, 
it works equally well when an extra $B$-particle is initially present in $\zeta$, and hence
in all cases where $\mathcal A^+=\{x\}$ and $\mathcal A^-=\varnothing$.
If instead $\mathcal A^+=\varnothing$ and $\mathcal A^-=\{x\}$, the same construction
yields a tracer that at all times follows an extra $B$-particle in $\zeta'$
or an extra $A$-particle in $\zeta$.

When $\Aa$ is not a singleton, the same construction applied to each element of $\Aa$
yields a collection of tracers.
Tracers originating from $\Aa^+$ are called \textit{$\oplus$-tracers}, and those originating from $\Aa^-$ are called \textit{$\ominus$-tracers}.
The only complication comes when a $\oplus$-tracer encounters a $\ominus$-tracer.
This may correspond to two extra particles of the opposite type, either both in $\zeta$ or both in $\zeta'$,
which annihilate each other. In this case the tracers are left with nothing to track,
and the two differences between $\zeta$ and $\zeta'$ that were tracked by them are no more.
It can also correspond to two extra particles of the same type, one in $\zeta$ and one in $\zeta'$,
both simultaneously annihilated by the same particle in both systems, again eliminating
two differences between $\zeta$ and $\zeta'$ and leaving the tracers nothing to track.
In these cases, we say that the two tracers are \emph{wandering} rather than \emph{active} from this point on,
but we extend their paths by independent rate-$\lambda_A$ random walks. (Note that
a $\oplus$- and $\ominus$-tracer can also encounter each other without the tracers becoming
wandering, for example when they track two extra $A$-particles, one in $\zeta$ and one in $\zeta'$,
but no $B$-particle encounters them before they move apart.)
We then have
  \begin{equation}\label{eq:diffA}
\begin{aligned}
    \zeta_t' - \zeta_t &= \sum_{\sigma \in \{+,-\}} \sum_{y \in \Aa^\sigma} \sigma \ind{X^y \text{ is active at time $t$}} \delta_{X_{t}^y}. 
\end{aligned}
\end{equation}

\begin{proof}[Proof of \thref{lem:monotonicity}]
  This lemma follows immediately from Lemma~3 in \cite{vicius_DLA}, which asserts that when
  extra $A$-particles are added to the graphical construction, the lifespans of all existing $A$-particles
  are at least as long as before, and the lifespans of all existing $B$-particles are no longer than
  before. (Lemma~3 in \cite{vicius_DLA} is itself an immediate consequence of the tracer construction and \eqref{eq:diffA}.)
\end{proof}

We are now ready to prove \thref{lem:finite.range.truncation}.
The idea is to consider DLAS  with full and truncated initial conditions and then
to bound the difference between the two systems using tracers.

\begin{proof}[Proof of \thref{lem:finite.range.truncation}]
We first provide the argument for the bidirected tree, then explain how to adapt it to the lattice and torus.  Let $B_r$ denote the ball of radius $r$ around the root in $\BT$ and let $r=\floor{ct}$ for 
a constant $c>1$ to be chosen later. 
Let $\zeta$ denote DLAS on $\BT$, and let $\zeta'$ denote the same process with particles
beyond graph distance~$r$ of the root removed from the initial conditions. Take $\zeta$ and $\zeta'$
to be coupled by the graphical construction described previously, and
for $y\in\BT\setminus B_r$, let $X^y=(X^y_u)_{u\geq 0}$ denote
the tracer path originating at $y$.

Let $\alpha_s(x)=\zeta_s(x)\1\{\zeta_s(x)>0\}$ and $\alpha_s'(x)=\zeta_s'(x)\1\{\zeta_s'(x)>0\}$,
the number of $A$-particles present at $x$ at time~$s$ in each system.
Since $\rho_s-\rho^r_s= \E[\alpha_s(\root)-\alpha'_s(\root)]$,
and $\abs{\alpha_s(\root)-\alpha'_s(\root)}$ is bounded by the number of tracers
at $\root$ at time~$s$,
  \begin{align}\label{eq:rho_diff}
    |\rho_s - \rho^r_s| \leq  \sum_{y \in \BT\setminus B_r} \P(X_{s}^y = \root ).
  \end{align}

  Now we bound this sum. Consider $y\in \BT\setminus B_r$ such that there is an oriented path in $\BT$ from $y$ to the root of length $u>r$. We claim that for such $y$ \commHL{and for all $1\le s\le t$}, 
  \begin{align}
    \P(X_s^y = \root) &\leq e^{-(u-t)^2/u}d^{-u}.
      \label{eq:BT.hitting.bound}
  \end{align}
  Indeed, as we mentioned when defining tracers,   
  since $\lambda_A,\lambda_B\leq 1$ we can couple $X^y$ with a rate~$1$ random walk $S=(S_v)_{v\geq 0}$ from
  $y$
  so that $X^y$ visits a subset of the sites traversed by $S$ up to a given time.
  For $S$ to reach $\root$ by time~$s\leq t$, it must make at least $u$ jumps by time~$s$,
  which occurs with probability at most $e^{-(u-t)^2/u}$ by \thref{lem:poisson_tail} 
  (note that $u\geq t$ since $c>1$).
  Also, $S$ must make its first $U$ jumps toward the root if it is ever to reach it, which
  has probability $d^{-u}$. Together, this proves \eqref{eq:BT.hitting.bound}.
   
  The number of $y$ such that there is an oriented path in $\BT$ from $y$ to the root of length $u$ is $d^u$ (see Figure~\ref{fig:tree}), and for $y$ such that there is no oriented path from $y$ to the root, we clearly have $\P(X_s^y = \root)=0$.
  Thus, splitting up the sum according to $u$ \commHL{and recalling $r=\lceil ct \rceil$}, it holds for $s\leq t$ that
  \begin{align}
     \sum_{y \in \BT\setminus B_r} \P(X_{s}^y = \root )
       &\leq \sum_{u=\ceil{ct}}^{\infty} e^{-(u-t)^2/u} \le e^{2t}\sum_{u=\ceil{ct}}^{\infty} e^{-u} \leq e^{-t} \label{eq:union}
  \end{align}
  for a sufficiently large choice of $c$, which gives the lemma for the bidirected tree.

  For the lattice $\mathbb{Z}^d$, let $B_r$ denote the closed $\ell^\infty$-ball of radius $r$ centered at $\root$.
  Define $\zeta$ as DLAS on $\ZZ^d$ and $\zeta'$ as the same with particles outside of $B_r$ removed
  from the initial conditions. As before,
  \begin{align}\label{eq:rho_diff2}
    \abs{\rho_s - \rho^r_s} \leq  \sum_{y \in \ZZ^d\setminus B_r} \P(X^y_s = \root ).
  \end{align}
  \commHL{The argument below holds whenever $1\le r\le 2t$, but the resulting bound will only be meaningful when $r\ge \sqrt{t}$.} 
  
  Consider $y\in \ZZ^d\setminus B_r$. As in the $\BT$ case, we bound 
  $\P(X^y_s=\root)$ by the probability of a rate~$1$ random walk $S=(s_v)_{v\geq 0}$ from $y$ hitting $\root$
  before time~$s$. Using a crude bound here, let $u=\norm{y}_\infty>r$ and choose a coordinate of $y$
  whose absolute value is $u$. The projection of $S$ onto this component is a rate~$1/d$ 1-dimensional
  random walk, and its probability of reaching $0$ by time $s\leq t$ is at most $e^{-u^2/4t}$
  by \thref{lem:SRW}. 
  Using this bound and letting  
  $b_u = |\{ x \in \mathbb Z^d \colon \|x\|_\infty = u\}| \leq C_du^{d-1}$, 
  the expected number of tracers started between distance $r+1$ and $2t$ from $\root$ that reach $\root$ by time $t$ is at most \commHL{(recall $r\le 2t$)}
\begin{equation}
\sum_{y\in B_{2t}\setminus B_r} \P(X^y_s = \root) \le \sum_{u=r+1}^{2t} b_u e^{-u^2/4t} \leq \sum_{u=r+1}^{2t} C_d u^{d-1} e^{-u^2/4t} = O(t^{d-1} e^{-r^2/4t}).  \label{eq:<2t} 
\end{equation}
The asymptotic bound is immediate when $d=1$ and for $d\geq 2$ is obtained by comparing the last sum in \eqref{eq:<2t} to the integral
\begin{align}
\int_r^{2t} u^{d-1} e^{-u^2/4t} &\leq (2t)^{d-2}	\int_{r}^{2t} u e^{-u^2/4t} du \\
			&= C' t^{d-1} \int_{r^2/4t}^{t} e^{-w} dw \leq C' t^{d-1} e^{-r^2/4t}. \label{eq:integral.compare}
\end{align}

 Now, we use \thref{lem:poisson_tail} to bound
\begin{align}
\sum_{y\in \ZZ^d\setminus B_{2t}} \P(X^y_s = \root)\leq \sum_{k=1}^\infty b_{2t+k} \exp\left(\f{-(t+k)^2}{2(2t+k)}\right) = O(t^d e^{-c't}).  \label{eq:>2t} 
\end{align}
The asymptotic claim follows from a similar approach as at \eqref{eq:integral.compare}.
Combining \eqref{eq:<2t} and \eqref{eq:>2t} with \eqref{eq:rho_diff} gives the claimed bound. A similar argument gives the analogous bound for $\mathbb T_{2r}^d$. 
\end{proof}

We now prove that DLAS on $\mathbb T_{2r}^d$ is comparable to DLAS on $\mathbb Z^d$. We start with an outline of the argument. First we couple the particle types and random walk paths at corresponding sites
of the torus and lattice.
Letting $F$ be the event that a particle started within distance $r$ of the root interacts with the boundary, we show that $\P(F)$ is small.  We then show that the density of particles at the root of the torus and of the lattice does not change much when all particles beyond distance $r$ are removed. With this removal, the processes on the torus and lattice are identical when $F^{c}$ occurs. And on the event $F$, the density of particles at the root can be easily controlled by comparing to DLAS systems with only $A$-particles present. Since $\P(F)$ is small, we obtain a good estimate on $|\rho_t - \bar \rho_t|$.

\begin{proof}[Proof of \thref{lem:torus}]

	Identify each site $\bar x \in \mathbb T_{2r}^d$ with $x \in \mathbb Z^d\cap (-2r,2r]^d$ in the canonical way that comes from viewing $\mathbb T_{2r}^d$ as a quotient space on the lattice with points in $(-2r,2r]^d$ as representatives of each equivalence class. 
  We now couple DLAS on $\ZZ^d$ and on $\mathbb T_{2r}^d$ using the graphical
  construction from \cite{vicius_DLA}. First,
  couple the initial configurations in $\mathbb T_{2r}^d$ and $\mathbb Z^d\cap (-2r,2r]^d$ to be the same. Using the natural identification of points outside of $(-2r,2r]^d$ to the equivalence class representative in $\mathbb T_{2r}^d$, couple the random walk instructions at $\bar x\in \mathbb T_{2r}^d$ and $x\in \mathbb Z^d\cap (-2r,2r]^d$ to be the same. For $x\in \mathbb Z^d\setminus (-2r,2r]^d$, let the initial configuration and instructions be generated independently. We will refer to sites in $\mathbb T_{2r}^d$ with some coordinate entry equal to $2r$ as \emph{boundary sites of $\mathbb T_{2r}^d$}.
	
	 Let $B_r = \{ \bar x \in \mathbb T_{2r}^d \colon \|x\|_{\infty} \leq r\}$. 
	Let $D_{\bar x}$ be the event that in time $t$ the {random walk instructions for the}  particle started at $\bar x \in B_r$ reaches a boundary site of $\mathbb T_{2r}^d$. 
	As $\bar x \in B_r$, the distance from $\bar x$ to a boundary site is at least $r$. Hence,  $\P(D_{\bar x})$ is bounded by the probability that a rate 1 simple random walk has displacement at least $r$ by time $t$. Each coordinate of the $d$-dimensional simple random walk on $\mathbb{Z}^{d}$ forms a rate $1/d<1$ simple random walk on $\mathbb{Z}$. Let $F=\cup_{\bar x \in B_r} D_{\bar x}$ be the event that some particle started inside of $B_r$ reaches a boundary site of the torus by time $t$. \thref{lem:SRW} and a union bound over the initial locations $\bar{x}\in B_{r}$ and $d$ coordinates give
		\begin{align}\P\left ( F \right) \le d|B_r| e^{ -r^2/\commHL{4}t} \le d(2r)^{d} e^{-r^2/\commHL{4}t}\le d(4t)^{d} e^{-r^2/\commHL{4}t}.\label{eq:torus_union_bound}	
		\end{align}
Note that \thref{lem:SRW} still applies to the rate $1/d$ random walk since the maximum of this random walk is dominated by that of a rate $1$ random walk.

Let $\rho^*_t$ and $\bar \rho ^*_t$ be the expected density of particles at $\root$ at time $t$ for the lattice and torus, respectively, with all particles initially beyond distance $r$ deleted (in the $\ell^\infty$-norm). By \thref{lem:finite.range.truncation}, \commHL{there are constants $C,c>0$} for which we have
\begin{align}
|\rho_t - \rho_t^*|, |\bar \rho_t - \bar \rho_t^*| \leq C t^d e^{-c r^2/t}. \label{eq:stars}	
\end{align}
Hence, it suffices to compare $\rho_t^*$ and $\bar \rho_t^*$. 

Let $Z$ and $\bar Z$ be the number of $A$-particles at the origin at time~$t$ 
on the lattice and torus, respectively, with the particles beyond distance $r$ deleted from the initial
configuration. We claim that 
$$Z\ind{F^c} = \bar Z \ind{F^c}$$
since, on the event $F^c$, the random walk paths are identical for corresponding particles in
the two coupled processes. It follows that
\begin{align}
|\rho_t^{*} - \bar \rho_t^{*}| 
					  &= \abs[\big]{\E[(Z- \bar Z)\ind{F}]}.
\end{align}
 The Cauchy-Schwartz inequality and the simple bound $\E[(Z-\bar Z)^2] \leq \E[Z^2] + \E[\bar Z^2]$ give  
\begin{align}
					  \abs[\big]{\E[(Z-\bar Z)\ind{F}]}&\leq \sqrt{ \E [Z^2] + \E [\bar Z ^2 ] } \sqrt{\P(F)}. \label{eq:CS}
\end{align}

It follows from \cite[Lemma~3]{vicius_DLA} that both $Z$ and $\bar Z$ are dominated by the counts of $A$-particles at the root in systems with an $A$-particle initially at every site.  
Call these dominating counts on the lattice and torus $Z'$ and $\bar Z'$, respectively. 
By symmetry of the underlying graphs, $\E [Z'] = 1 = \E[\bar Z']$ for all $t\geq 0$. Using this fact and expressing $Z'$ and $\bar Z'$ as sums of independent indicators for whether or not the particle started at each site is at $\root$ at time $t$, it is straightforward to prove that $\E [(Z')^2], \E[(\bar Z')^2] \leq 2$. In full detail, letting $p_x$ be the probability a particle started at $x$ is at $\root$ at time $t$, expanding $\E[(Z')^2]$ and $(\E[ Z'])^2$ gives
$$\E[(Z')^2 ]= 2\sum_{x,y \in \mathbb Z^d, x \neq y} p_xp_y + \sum_{x \in \mathbb Z^d} p_x \leq (\E[Z'])^2 + \E[Z'] = 2.$$
Proceeding with similar reasoning also gives that $\E[(\bar Z')^2] \leq 2.$ 

It follows from \eqref{eq:CS} that 
\begin{align}|\rho_t^{*} - \bar\rho_t^{*}| \leq 2\sqrt{\P(F)}. \label{eq:rho-barrho}\end{align} 
Using \eqref{eq:torus_union_bound}, we have $\P(F) \leq C_3 t^d e^{-c_4 r^2 /t}$ with $c_4, C_3>0$. Applying this to \eqref{eq:rho-barrho} and then using \eqref{eq:stars} with the triangle inequality  gives
$$|\rho_t - \bar\rho_t| \leq C' t^{d} e^{- c' r^2 / t}$$
for some constants $c',C'>0.$
\end{proof}

\subsection{A variation on tracers and the proof of Lemma~\ref{lem:sequential.process}} \label{sec:sequential}
 
We now give another construction of DLAS designed to relate it to the sequential
process defined in Section~\ref{sec:lemmas}. It resembles the dragged tracer construction
from \cite[Section~4.1]{vicius_DLA} but is not quite the same (see \thref{rmk:CRS.comparison}).

We assume now that $\lambda_A=1$ and $\lambda_B=0$. We will define a particle system with
three types of particles: $B$-particles, $A$-tracers, and $B$-tracers. As we will see
in \thref{prop:tracer.system}, if we view the $A$-tracers as $A$-particles and ignore
the $B$-tracers altogether, the resulting system will have the law of DLAS.

We now describe this process, which we call the \emph{$A$/$B$-tracer system}. 
At time~$0$, the system consists of possibly infinitely
many $B$-particles, no more than one per site, as well as a finite number of particles we call 
$A$-tracers, which we number $1,\ldots,n$.
The $k$th $A$-tracer is assigned a rate-$1$ random walk $S^k = (S^k_t)_{t\geq 0}$
as its \emph{putative trajectory}.
The $A$-tracers initially follow these trajectories, while $B$-particles do not move.
When an $A$-tracer jumps onto a $B$-particle, the $B$-particle is annihilated and the $A$-tracer
becomes a $B$-tracer and halts. If the $j$th $A$-tracer jumps onto the $k$th $B$-tracer
with $j<k$, then the $A$-tracer becomes a $B$-tracer and halts, while the $B$-tracer becomes an $A$-tracer
and resumes following the path $S^k$ from where it left off when it became a $B$-tracer;
if $j>k$ then no interaction occurs.

We claim that under these dynamics, there is no need to assign a braveness to each particle
because particles can never achieve a configuration where it is ambiguous which particles should react.
Under these interaction rules and our assumptions on the initial configuration, 
there is at most one $B$-particle or $B$-tracer on a site at all times.
Since $A$-tracers move in continuous time, almost surely only a single $A$-tracer jumps
onto a $B$-particle or $B$-tracer at a time. 
And while an $A$-tracer may jump onto a site that contains $A$-tracers and a $B$-tracer,
the $A$-tracers already present will have indices greater than the $B$-tracer's, and only
the newly arrived $A$-tracer may interact with the $B$-tracer.
Also, since the system contains only finitely many moving particles, there is no question
that the construction is well defined.
We record two immediate observations for later reference:
\begin{lemma}\thlabel{lem:AB.basics}
  In the $A$/$B$-tracer system:
  \begin{enumerate}[(a)]
     \item a site that does not start with a $B$-particle will never contain $B$-particles or $B$-tracers;
     \item a site that initially contained a $B$-particle still contains the $B$-particle if it has not been visited
by $A$-particles yet; otherwise it contains a $B$-tracer whose index is the lowest of all tracers
that have visited the site up to that point.\label{i:lowest.visited}
  \end{enumerate}
\end{lemma}

As we said before, the $A$/$B$-tracer system gives an alternative construction of DLAS:
\begin{proposition}\thlabel{prop:tracer.system}
  Define $\zeta_t(x)$ as the number of $A$-tracers minus the number of $B$-particles
at $x$ at time~$t$ in the particle system defined above. Then $\zeta$ is a DLAS
with $\lambda_A=1$ and $\lambda_B=0$.
\end{proposition}
\begin{proof}
  Simply observe that $\zeta_t$ is a Markov process with the same dynamics as DLAS.
\end{proof}

\begin{remark}\thlabel{rmk:CRS.comparison}
  The dragged tracer construction of \cite[Section~4.1]{vicius_DLA} is like the system of tracers (also
  from \cite{vicius_DLA}) described in Section~\ref{sec:tracers}, except that tracers have their own
  prespecified paths, rather than simply following the prespecified paths of the particles
  they are tracing. The $A$/$B$-tracer system resembles the dragged tracer construction that
  would result from adding a collection of $A$-particles to a DLAS that has only $B$-particles.
  But the two constructions differ in that in our system, an $A$-tracer can interact with a $B$-tracer,
  whereas two $\oplus$-tracers in \cite{vicius_DLA} never interact.
  
  The $A$/$B$-tracer system could also be defined when $\lambda_B>0$.
  We do not do so here because we do not need it, and the possibility of a tracer jumping
  onto a site containing multiple particles and tracers of the opposite type adds
  some technical difficulty. A closely related construction appears in \cite{bahl2021diffusion}.
  The system defined there has exactly two tracers and also $A$-particles that interact
  with the tracers by similar rules. It still has $\lambda_B=0$, but it allows multiple $B$-particles 
  to start on the same site and for
  two $A$-tracers to jump on a site simultaneously in discrete time.
    
  We defined $\zeta$ in terms of the $A$/$B$-tracer system by having it count $A$-tracers
  as $A$-particles and ignore $B$-tracers. If we instead define $\zeta^k$ to count tracers
  $1,\ldots,k$ as $A$-particles when in state~$A$, to count tracers $k+1,\ldots,n$
  as $B$-particles when in state~$B$, and to ignore the tracers in other states, then $\zeta^k$
  is a DLAS that initially has $k$ $A$-particles. This gives a coupling of the sequence
  of systems resulting from successively adding $A$-particle one at a time.
  We do not need this for this paper, though it is used in \cite{bahl2021diffusion}.
\end{remark}

We now modify the $A$/$B$-tracer system by killing each tracer particle when its elapsed time
spent in state~$A$ reaches $t$. We call the resulting system the \emph{$t$-killed $A$/$B$-tracer system}.
In more detail,   
at a given time $s$, each tracer particle in the $A$/$B$-tracer system has had a finite number
of interactions that cause it to switch from an $A$-tracer to a $B$-tracer or vice versa. 
Let $L_k(s)$ denote the
combined length of the time periods up to $s$ that the $k$th tracer has spent as an $A$-tracer,
so that its position at time~$s$ is $S^k_{L_k(s)}$. At time $U_k:=\inf\{s\colon L_k(s)\geq t\}$,
the $k$th tracer is killed and hence removed from the system. Observe that the $k$th tracer may enter state~$B$ 
and never leave it before spending time~$t$ in state~$A$, in which case $U_k=\infty$.
We define $T_k=t$ if $U_k<\infty$ and $T_k=\max\{L_k(s)\colon s\geq 0\}$ if $U_k=\infty$.
Thus the $k$th tracer traces out the path $(S^k_s)_{0\leq s \leq T_k}$ over the lifetime
of the process, pausing whenever it is in state~$B$.
If $t=\infty$, then we define $U_k=\infty$ and $T_k=\sup\{L_k(s)\colon s\ge 0\}$, with $T_k=\infty$
if $L_k(s)$ is unbounded.
The resulting process is the original $A$/$B$-tracer system and as in the $t<\infty$
case, the $k$th tracer follows the path $(S^k_s)_{0\leq s \leq T_k}$ over its lifetime.

\begin{proof}[Proof of \thref{lem:sequential.process}]
  Consider the sequential process run for time~$t\leq\infty$.
  For now, take $H$ to be finite and
  let $n$ be the number of $A$-particles in the set $H$.
  For $k\in\{1,\ldots,n\}$, let $S^k=(S^k_s)_{s\geq 0}$ denote the random walk followed
  by the $k$th $A$-particle.
  Let $T'_k$ be the length of time spent walking by the $k$th $A$-particle in the sequential process
  before annihilation, with $T'_k=t$ if it is never annihilated.
  Thus the path of the $k$th particle in the sequential process is $(S^k_s)_{0\leq s\leq T'_k}$.

  Now, we consider the $t$-killed $A$/$B$-tracer system with initial conditions 
  corresponding to the sequential
  process, that is, with $A$-tracers initially at locations in $H$ containing
  $A$-particles and with $B$-particles initially at all sites in $G$ with $B$-particles.
  We give the tracers the same ordering as in the sequential system, and we couple
  the sequential process with the $t$-killed $A$/$B$-tracer system by using 
  the same collection of random walks $\{S^1,\ldots,S^n\}$ for both processes.
  Recall that $T_k$ was defined so that $(S^k_s)_{0\leq s\leq T_k}$ is the path traced
  out by the $k$th tracer in this system.
    
  We claim that $T_k=T'_k$ for all $k$, and hence that the paths traced out in the sequential process
  and $t$-killed $A$/$B$-tracer systems are the same. 
  The rest of the lemma will follow easily once this claim is proved, since the original and $t$-killed
  $A$/$B$-tracer systems do not differ until at least time~$t$.
  
  We prove the claim by induction on $k$.
  To prove that $T_1=T'_1$, we observe that the first $A$-tracer enters state~$B$ on colliding
  with any $B$-particle or $B$-tracer. Since every site that starts with a $B$-particle
  contains a $B$-particle or a $B$-tracer at all times, the first $A$-tracer becomes
  a $B$-tracer exactly when it visits a site that originally contained a $B$-particle.
  Since it has the lowest index, it can never turn back to an $A$-tracer. Thus $T_1$
  is the first time that $S^1$ visits a site that originally contained a $B$-particle,
  or $T_1=t$ if $S^1$ does not visit such a site by time~$t$. This is exactly the
  description of $T'_1$ as well, proving that $T_1=T'_1$.
  
  Now, suppose that $T_j=T'_j$ for $j=1,\ldots,k-1$, and we will show that $T_k=T'_k$.
  First, observe that $T_k$ and $T'_k$ must either be jump times of the walk $S^k$
  or be equal to $t$. For $T_k$, this is because if $T_k<t$, then the $k$th tracer
  entered state~$B$ at some time $u$ satisfying $T_k=L_k(u)$ and never reentered state~$A$.
  Since an $A$-tracer enters state~$B$ only upon jumping, $T_k$ is a jump time of $S^k$.
  Similarly, for $T'_k$ it is because a particle in the sequential process
  is annihilated only at a jump time of $S^k$.
  
  Thus, to show $T_k=T'_k$
  it will suffice to show that for jump times $r$ of $S^k$ satisfying $r<t$, we have
  $T_k=r$ if and only if $T'_k=r$.
  We can also assume that the first visit by $S^k$ to $S^k_r$ occurs at time~$r$,
  since in the $t$-killed $A$/$B$-process the $k$th $A$-tracer can
  only enter state~$B$ on its first visit to a site, and in the sequential process
  an $A$-particle can only be annihilated on its first visit to a site.
  For such a jump time $r<t$ with $x:=S^k_r$, we claim that $T_k=r$ if and only if 
  the following conditions hold:
  \begin{enumerate}[(i)]
    \item $T_k\geq r$; \label{i:tger}
    \item there is initially a $B$-particle at $x$ in the $t$-killed $A$/$B$-tracer system; and \label{i:tinit}
    \item for $j\in\{1,\ldots,k-1\}$, the path $(S^j_u)_{0\leq u\leq T_j}$
      does not contain $x$.\label{i:tlowers}
  \end{enumerate}
  Suppose the above conditions hold. Since $T_k\geq r$, the $k$th tracer survives long enough
  to jump to site~$x$ at some time~$s$ with $r=L_k(s)$. Since there is initially a $B$-particle at $x$
  and the tracers of index smaller than $k$ do not visit $x$ by \ref{i:tlowers}, 
  there remains a $B$-particle or $B$-tracer of index greater than $k$ by \thref{lem:AB.basics}\ref{i:lowest.visited} that the $k$th tracer will
  jump onto. The $k$th tracer then enters state~$B$ and never leaves,
  since no $B$-particle or $B$-tracer of index smaller than $k$ visits $x$. Hence $T_k=r$.
  Conversely, suppose $T_k=r$. Then the $k$th tracer jumps onto site $x$ at some time~$s$ with $r=L_k(s)$,
  enters state~$B$, and never leaves. Clearly \ref{i:tger} holds.
  By \thref{lem:AB.basics}\ref{i:lowest.visited}, condition~\ref{i:tinit} holds, and no tracer of index
  less than $k$ visits $x$ before time~$s$. And since the $k$th tracer never leaves state~$B$,
  no tracer of index less than $k$ visits $x$ after time~$s$ either, showing that \ref{i:tlowers}
  holds.
  
  Next, observe that for a jump time $r$ of $S^k$ with $x:=S^k_r$ and with $S^k$ assumed
  not to visit $x$ until time $r$, we have $T'_k=r$ if and only if 
  the following conditions hold for the $k$th $A$-particle in the sequential process:
  \begin{enumerate}[(i')]
    \item $T'_k\geq r$, that is, the particle has not been annihilated before time~$r$;  \label{i:sger}
    \item there is initially a $B$-particle at $x$ in the sequential process; and \label{i:sinit}
    \item none of the first $k-1$ $A$-particles in the sequential process visited $x$,
      so for $j\in\{1,\ldots,k-1\}$ the path $(S^j_u)_{0\leq u\leq T_j'}$  does not contain $x$.
      \label{i:slowers}
  \end{enumerate}
  By our coupling, \ref{i:tinit} holds if and only if \ref{i:sinit} does. By the inductive hypothesis,
  \ref{i:tlowers} holds if and only if \ref{i:slowers} does.
  And by doing induction on $r$ starting with the earliest jump time, we can assume
  that \ref{i:tger} holds if and only if \ref{i:sger} holds.
  This proves that $T_k=T'_k$, advancing the induction and proving that $T_k=T'_k$ for all $k$.
  
  Let \commHL{$V_s^{(t)}$} be the occupation time of the root by $A$-tracers up to time~$s$ in the $t$-killed $A$/$B$-tracer
  system. Then
  \begin{align*}
    V_t^{(t)} &= \sum_{k=1}^n \int_0^t\1\{\text{the $k$th \commHL{$t$-killed} tracer is in state~$A$ at $\root$ at time~$s$}\}\,ds\\
        &\leq \lim_{s\to\infty}V_{s}^{(t)} = \sum_{k=1}^n \int_0^{T_k}\1\{S^k_s=\root\}\,ds.
  \end{align*}
  Since the $t$-killed $A$/$B$-tracer system matches the $A$/$B$-tracer system at least up to time~$t$,
  by \thref{prop:tracer.system} the random variable $V_t^{(t)}$ is distributed as the occupation time  $V_{t}$ of the root
  in DLAS with initial conditions as in the statement of the lemma.
  Let $V'_t$ be the occupation time of the root in the sequential process, i.e.,
  \begin{align*}
    V'_t = \sum_{k=1}^n \int_0^{T'_k}\1\{S^k_s=\root\}\,ds.
  \end{align*}
  Since $T_k=T'_k$ for all $k$, this proves that $V_t\commHL{\preceq} V'_t$.
  
  Now, suppose that $H$ is infinite. Let $H_n$ consist of the first $n$ vertices in $H$
  in the given ordering, and let $V_{t,n}$ and $V'_{t,n}$ be the occupation times of the root
  in DLAS and the sequential process, respectively, when $A$-particles outside of
  $H_n$ are removed from the initial configuration.
  By the case of this lemma already proven, $V_{t,n}\preceq V'_{t,n}$.
  When $V_{t,n}$ is defined in terms of the graphical construction from \cite{vicius_DLA},
  it increases in $n$ by \cite[Lemma~3]{vicius_DLA} and converges almost surely
  to $V_t$. By definition of the sequential process, $V'_{t,n}$ converges upwards to $V'_t$.
  Thus $V_{t,n}\to V_t$ and $V'_{t,n}\to V'_t$ in law as $n\to\infty$ and $V_{t,n}\preceq V'_{t,n}$
  for all $n$, which together prove that $V_t\preceq V'_t$.
\end{proof}

\subsection{Polarized construction of DLAS and proof of Lemma~\ref{lem:polarized_DLAS}}
\label{subsection:polarized_construction}

We give one last construction of DLAS that we call the \emph{polarized system}.
As in the previous section's construction, the system has tracers that can be either in state~$A$
or state~$B$---we call them $A$-tracers or $B$-tracers depending on their current state---as 
well as $B$-particles. 
We assume that there are only finitely many tracers in the initial configuration.
Each tracer starts in state~$A$ and is given
a rate-1 random walk as its putative trajectory, and it follows it while in state~$A$.
The $B$-tracers and $B$-particles are immobile.

Each tracer and $B$-particle in the polarized system is assigned a \emph{polarity}, 
either \emph{positive} or \emph{negative}, as part of the initial configuration.
In the following interaction rules, these polarities play a role similar to the indices 
in the $A$/$B$-tracer system:

\begin{enumerate}[label = (\alph*)]
  \item If an $A$-tracer jumps onto a $B$-particle, then the $B$-particle is annihilated and 
    the $A$-tracer enters state~$B$\commHL{, regardless of their  polarities.} (The same occurs in the $A$/$B$-tracer system.)  
    \label{i:polara} 
  \item If an $A$-tracer jumps onto a $B$-tracer of opposite polarity at site~$v$, then 
    whichever tracer matches the polarity of the $B$-particle initially at $v$ enters (or remains in) 
    state~$B$. The other tracer enters (or remains in) state~$A$ and continues along its putative trajectory.
    \label{i:polarb}
  \item If an $A$-tracer jumps onto a $B$-tracer of the same polarity at site~$v$, then
    whichever tracer's putative trajectory contains $v$ at the earliest time enters (or remains in) state~$B$.
    The other tracer enters (or remains in) state~$A$ and continues along its putative trajectory.
    \label{i:polarc}
\end{enumerate}
As with the $A$/$B$-tracer system, under these rules there is exactly one $B$-particle or $B$-tracer
at all times on each site that initially holds a $B$-particle. When an $A$-tracer jumps onto a site with
a $B$-tracer, the site  contains no other $A$-tracers or it contains $A$-tracers that do
not interact with the $B$-tracer, and there is no ambiguity about which $A$-tracer can
interact with the $B$-tracer. By our assumption of having only finitely many tracers,
the system is a Markov chain on a countable state space and there is no question
as to its existence.

Now, we define the
\emph{positive} and \emph{negative DLAS} via the graphical
construction from \cite{vicius_DLA} and then couple these DLAS with the polarized system.
We define the \emph{positive DLAS} as follows.
For every site in the polarized system that starts with a positive $A$-tracer, the positive DLAS
starts with an $A$-particle. For every site in the polarized system that starts with a positive $B$-particle,
the positive DLAS starts with a $B$-particle. At all other sites the positive DLAS initially has
no particles. Each $A$-particle in the positive DLAS is given the same putative trajectory
as the corresponding $A$-tracer in the polarized system. Since $\lambda_A=1$ and $\lambda_B=0$ for this
DLAS and we start with at most one $B$-particle per site, we do not need to assign a braveness
to our particles. The \emph{negative DLAS} is defined the same way, but its initial configuration
matches up with the negative particles in the polarized system.
The positive and the negative DLAS are then two DLAS, independent conditional on their initial configurations,
both coupled with the polarized system.

\begin{figure}
	\begin{center}
		\includegraphics[width = 1 \textwidth]{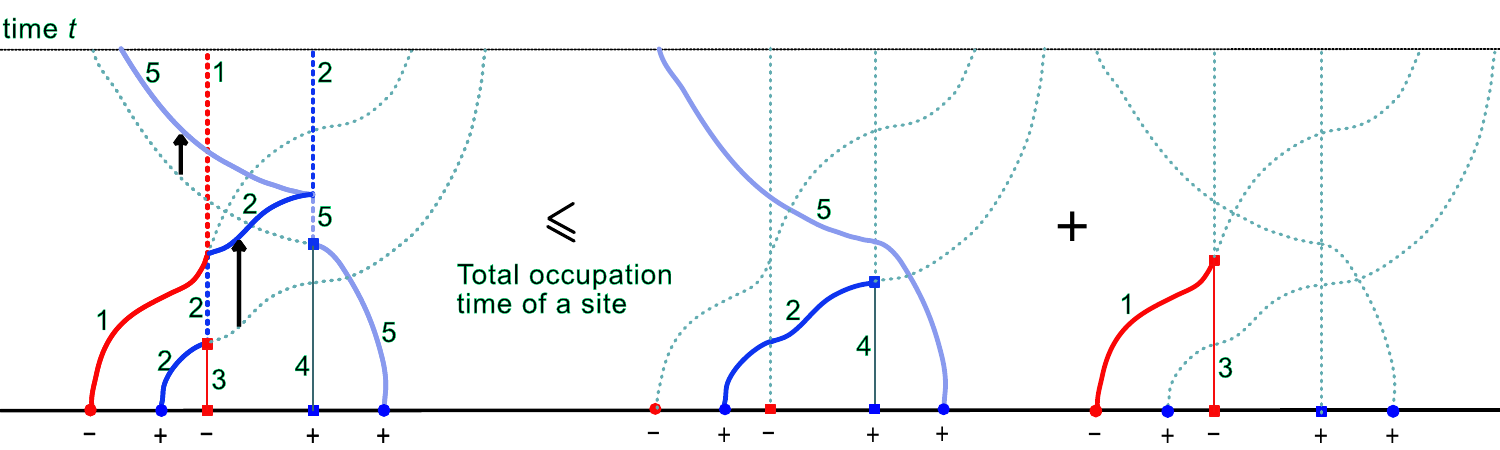}
	\end{center}
	\caption{A realization of the polarized system (left) with three initial $A$-tracers and two $B$-particles, and the corresponding behavior of the positive (middle) and negative (right) DLAS processes. 
 Dots and squares represent initial $A$-tracers and $B$-particles, respectively. In the left system, solid curves are trajectories of $A$-particles, solid vertical lines trajectories of $B$-particles, dotted curves are putative trajectories of $A$-tracers, and dotted vertical lines are trajectories of $B$-tracers. There are five trajectories in the left system and they are marked by numbers from 1 through 5. In the middle and the right systems, solid curves and vertical lines represent trajectories of $A$- and $B$-particles, respectively. 
      Blue and red indicate that the original source of the path was the putative trajectory of a 
  positive or negative particle, respectively. Two different shades of blue are used to
  disambiguate the two putative trajectories of positive $A$-tracers.
  Green dotted curves indicate portions of putative trajectories not used in their respective
  system (but in the polarized system on left some are used after a time shift, indicated by black upward arrows).
 At each time, every putative trajectory is traversed no further in the polarized system (left) than in each of the positive (middle) and negative (right) DLAS, so the occupation time of the origin by a given time $t$ in the polarized system is at most the sum in the corresponding positive and negative  DLAS.  
	}
	\label{fig:DLAS_polarized_pf}
\end{figure}

Our first claim is that if we ignore $B$-tracers in the polarized system, 
we obtain DLAS with the initial conditions given by the $A$-tracers
and $B$-particles (recall that in the polarized system all tracers
start in state~$A$).
 The statement and proof are essentially
the same as for \thref{prop:tracer.system}.
\begin{proposition} \thlabel{prop:path-lifting.valid} 
Let $\zeta_t(v)$ be the count of $A$-tracers minus the count
of $B$-particles on site~$v$ at time~$t$ in the polarized system. 
Then $\zeta_t$ is DLAS with $\lambda_A=1$ and $\lambda_B=0$.
\end{proposition}
\begin{proof}
  When we are blinded to the polarity of the particles and to $B$-tracers, the interaction
  between $A$- and $B$-tracers has no effect on $\zeta_t$ except replacing a particle's future random walk trajectory
  with an independent one.
  And when an $A$-tracer moves onto a $B$-particle, the $A$-tracer
  enters state~$B$ and the $B$-particle is destroyed, and neither is counted
  anymore by $\zeta_t$.  
  As in \thref{prop:tracer.system}, the process $\zeta_t$
  is then a Markov chain with the same interaction rules as DLAS
  with $\lambda_A=1$ and $\lambda_B=0$.
\end{proof}

\begin{proof}[Proof of \thref{lem:polarized_DLAS}]
  First we suppose that the given DLAS contains only finitely many $A$-particles.
  Form a polarized system by assigning polarities according to the signs of particles in the given DLAS. 
  This polarized system is coupled with the positive and negative DLAS, made up
  respectively of the positive and negative particles only from the given DLAS.
  By \thref{prop:path-lifting.valid}, the polarized system provides a construction
  of the given DLAS. Thus we obtain a coupling of the
  given DLAS together with the two DLAS, which are independent conditional on their initial configurations,
  consisting of the negative and positive particles.

  Let $S^v=(S^v_t)_{t\geq 0}$ denote the putative trajectory of the $A$-tracer
  initially at $v$, for all sites~$v$ initially containing $A$-tracers in the polarized
  system. With similar notation as in the $A$/$B$-tracer system, we define
  $L_v(t)$ as the combined time spent by the tracer initially at $v$ in state~$A$
  up to time~$t$, so that its position at time~$t$ is $S^v_{L_v(t)}$.
  
  Suppose that $v$ initially contains an $A$-tracer in the polarized system
  (or equivalently that $v$ initially contains an $A$-particle in the given DLAS).
  Let $T_v\leq\infty$ denote the time of annihilation of the $A$-particle in the
  corresponding negative or positive DLAS, depending on the tracer's polarity.
  We will argue that $L_v(t)\leq T_v$ for all $t\geq 0$, that is, a tracer in the polarized
  process never travels farther along its putative trajectory than the corresponding
  $A$-particle does in the positive  or negative DLAS.
  
  To prove this claim, we must show that if the particle starting
  at $v$ in the positive or negative DLAS is annihilated upon visiting $u$, then
  in the polarized process the tracer starting at $v$ will permanently enter state~$B$
  upon visiting $u$. The idea is that interaction rules \ref{i:polarb} and \ref{i:polarc} of
  the polarized process ensure that this tracer has priority over all other tracers to be in
  state~$B$.
  To make this argument precise, we consider the first time that it fails and derive
  a contradiction.  That is,
  suppose that our claim is false, and let $t$ be the infimum of all times $s$ such that
  there exists a vertex $v$ with $L_v(s)> T_v$.
  Since the system has only finitely many $A$-particles and $T_v>0$ for all $v$, we have $t>0$.
  Let $v$ be a site such that $L_v(t+\epsilon)>T_v$ for all $\epsilon>0$, and let $u=S^v_{T_v}$.
  As a shorthand, we write $\aaa$ to refer to the tracer in the polarized process
  initially at site~$v$.
  For the sake of exposition assume that $\aaa$ is positive.
  We write $\aaa^+$ as a shorthand for the particle in the positive DLAS initially at $v$,
  which has the same putative trajectory as $\aaa$.
  To summarize the set-up, particle~$\aaa^+$ is annihilated upon jumping onto
  site~$u$ at time~$T_v$, but tracer~$\aaa$ is on $u$ at time~$t$ with $L_v(t)=T_v$
  in state~$A$ and is about to continue on its putative trajectory past the point that $\aaa^+$
  was annihilated.
  Our job is to find a contradiction.
  
  When $\aaa$ jumps onto site~$u$, at some time $t'\leq t$, we claim that it enters state~$B$.
  Since $\aaa^+$ is annihilated at $u$, we know that initially there is a positive $B$-particle
  at $u$ in the positive DLAS and in the polarized process. By the dynamics of the system,
  there is a $B$-particle or $B$-tracer at $u$ at all times in the polarized process.
  If the particle on $u$ at time~$t'$ is a $B$-particle, then $\aaa$ enters state~$B$.
  If the particle on $u$ at time~$t'$ is a negative $B$-tracer, then $\aaa$ enters state~$A$ by the rules of the system.
  And the particle on $u$ cannot be a positive $B$-tracer: If it were, then some other positive tracer
  particle has visited $u$ before time~$t'$. By our minimality assumption for $t$,
  this particle traverses no more of its putative trajectory than its corresponding $A$-particle
  in the positive DLAS. But this is a contradiction, since then this $A$-particle in the positive DLAS 
  visits $u$ before time~$t'$, which would mean that $\aaa^+$ was not the first $A$-particle
  in the positive DLAS to visit $u$.
  
  Now, we argue that after entering state~$B$ at time~$t'$, particle $\aaa$ never returns
  to state~$A$. By the dynamics of the polarized process, it can only return to state~$A$ if
  it is visited by
  another positive $A$-tracer whose putative trajectory contains $u$ earlier than time $L_v(t)=T_v\leq t$.
  But again, by our minimality assumption for $t$, this would imply that some $A$-particle in
  the positive DLAS traverses this putative trajectory up to this visit to $u$, which contradicts
  $\aaa^+$ being the first $A$-particle in the positive DLAS to visit~$u$.
  
  Thus we have shown that $\aaa$ enters state~$B$ at some time $t'\leq t$ and never returns to state~$A$.
  This contradicts our assumption that $\aaa$ is in state~$A$ at time~$t$, completing our
  proof that $L_t(v)\leq T_v$ for all $v$.
  
  The lemma itself now follows easily. Let $v$ be a site initially containing
  an $A$-tracer in the polarized process. Since we have shown that $L_t(v)\leq T_v$
  and clearly $L_t(v)\leq t$, a tracer in the polarized process
  traverses at most length $\min(t, T_v)$ of its putative trajectory in time~$t$.
  On the other hand, the $A$-particle initially at $v$ in the positive or negative DLAS
  traverses exactly length $\max(t, T_v)$ of its putative trajectory in time~$t$.
  Thus, the time spent at the root by any $A$-tracer in the polarized process up to time~$t$
  is at most the time spent by the corresponding $A$-particle in the negative or positive DLAS,
  establishing that $V_t\leq V_t^+ + V_t^-$ under this coupling.
  
  Finally, we extend the theorem to initial configurations with infinitely many $A$-particles
  by a similar argument as in the proof of \thref{lem:sequential.process}.
  Enumerate the vertices of the graph in any way, and let $V_{t,k}$ denote the occupation
  time of the root up to time~$t$ in the given DLAS 
  with all $A$-particles beyond the first $k$ vertices of the graph
  removed from the initial configuration. Let $V^-_{t,k}$ and $V^+_{t,k}$ be the analogous
  quantities for the negative and positive DLAS. By monotonicity of the graphical
  construction \cite[Lemma~3]{vicius_DLA}, these random variables converge in law to $V_t$, $V^-_t$,
  and $V^+_t$, respectively, as $k\to\infty$. 
  By the case of the lemma already proven, $V_{t,k}\preceq V_{t,k}^-+V_{t,k}^+$, which together
  with the convergence proves that $V_t\preceq V_t^-+V_t^+$. And last, a similar limit argument
  proves the lemma in the $t=\infty$ case.  
\end{proof}

\section{Proof of main results on the integer lattice}
\label{sec:lattice}
In this section we prove our main results for DLAS on the integer lattice, which are stated in Theorems \ref{thm:LB}, \ref{thm:UB}, \ref{thm:EV_LB}, and \ref{thm:EV_UB}.
{As defined in Section~\ref{sec:notation}, we use $D(H_0)$ to denote the number of
$A$-particles minus the number of $B$-particles 
on a finite subgraph $H_0$ in the initial configuration of a given instance of DLAS.}

\subsection{A lower bound on \texorpdfstring{$\rho_t$}{rho\_t} for DLAS on \texorpdfstring{$\mathbb{Z}^{d}$}{Z\textasciicircum d}} \label{sec:LB}

We first derive \thref{thm:LB} from Lemma~\ref{lem:torus}.

\begin{proof}[Proof of Theorem~\ref{thm:LB}.]
	
	Let $r = \lceil C \sqrt{ t \log t} \rceil$, and let $\bar \rho_t$ be the density of $A$-particles at time~$t$ on $\TT_r^d$ as in \thref{lem:torus}. 
	By this lemma, we may choose $C$ large enough so that
	$$|\rho_t - \bar \rho_t| \lesssim (t \log t)^{-d/4}.$$ 
It then suffices to show that
	\begin{align}
		\bar \rho_t \gtrsim (t \log t)^{-d/4}. \label{eq:rho'}
	\end{align}

	Let $D = D(\TT_r^d)$ denote the difference between the initial number of $A$- and $B$-particles on $\mathbb T_r^d$.
  We can express $D$ as the sum of $|T_r^d|$ i.i.d.\ random variables taking value in $\{-1,1\}$ with mean 0. 
The central limit theorem and the fact that $| \mathbb{T}_{r}^{d}| \asymp (t \log t)^{d/2}$ guarantee that there is a constant $c>0$ so that
	\begin{align}
	\inf_{t\ge 0} \; \P\left( D > (t \log t)^{d/4}\right) > c. \label{eq:clb}
	\end{align}
Notice that on the event $D>(t \log t)^{d/4}$, there are at least $(t\log t)^{d/4}$ $A$-particles
on the torus that survive forever. Hence the expected number of $A$-particles
in the system at time~$t$ is at least $c(t\log t)^{d/4}$. By translation invariance,
\begin{align*}
  \bar\rho_t &\geq \frac{c(t\log t)^{d/4}}{\abs{\mathbb{T}_R^d}}\geq c'(t\log t)^{-d/4}.
\end{align*}
This completes the proof.
\end{proof}

\subsection{A lower bound on the critical exponent in low dimension} \label{sec:EV_LB}

The idea for the proof of Theorem~\ref{thm:EV_LB} was shared with us by Michael Damron. To give a lower bound on the critical exponent for $\mathbb{E}_p V$, we replicate the proof of Theorem~\ref{thm:LB} for $p<1/2$ using a more refined estimate in
place of the central limit theorem. 

\begin{lemma}\thlabel{lem:devr}
	Let $D = D(\TT_r^d)$, the discrepancy of the initial configuration of DLAS on the torus.
   Then there exists an absolute constant $c_{1}>0$ so that the following implication holds
   for all $1/4\leq p<1/2$:  
	\begin{equation}\label{eq: CE_1}
	2r \le  \biggl( \frac{c_1}{1-2p} \biggr)^{2/d}  \quad \Longrightarrow \quad \mathbb{P}_p\bigl(D \geq c_1(2r)^{d/2}\bigr) \geq c_1.
	\end{equation}
\end{lemma}
\begin{proof}
  Let $n=(2r)^d$, the volume of $\TT_r^d$.
  The random variable $D$ is the sum of $n$ independent random variables
  taking values $1$ and $-1$ with probability $p$ and $1-p$, respectively.
  Thus $\E D = -n(1-2p)$ and $\sigma^2 := \Var D = 4np(1-p)$.
  The idea of the proof is that the condition $2r\leq\bigl(c/(1-2p)\bigr)^{2/d}$
  ensures that $D$ has mean and standard deviation on the order of $\sqrt{n}=(2r)^{d/2}$,
  and hence $\P_p(D\geq c(2r)^{d/2})$ is bounded from below.
  We make this precise now using a quantitative CLT.
  
  Let $c>0$ be a small constant to be determined later.
  Suppose that $2r\leq\bigl(c/(1-2p)\bigr)^{2/d}$.
  It follows that $1-2p\leq c/\sqrt{n}$. From our assumption that $1/4\leq p<1/2$,
  we have $\sigma^2\geq 3n/4$. Thus
  \begin{align*}
    \E D + \frac{4c\sigma}{\sqrt{3}} \geq c\sqrt{n}=c(2r)^{d/2}.
  \end{align*}
  With $Z$ a standard normal random variable, we obtain
  \begin{align*}
    \P_p\bigl( D\geq c(2r)^{d/2}\bigr) \geq \P_p\biggl(D \geq \E D + \frac{4c\sigma}{\sqrt{3}}\biggr)
    \geq \P\Bigl( Z\geq 4c/\sqrt{3} \Bigr) - \frac{C}{\sqrt{n}}
  \end{align*}
  for some absolute constant $C$ by the Berry--Esseen CLT. \commHL{Take $c=c'$ sufficiently small so that $\P\bigl( Z\geq 4c'/\sqrt{3} \bigr)\ge 1/4$. Then the lower bound above is at least $1/8$ for all $n\ge (8C)^{2}$ since  $C/\sqrt{n}\le 1/8$. By consiering the finite number of small values of $n$ not covered by the previous bound, we get $\P_p\bigl( D\geq c'(2r)^{d/2}\bigr)\geq c''$ for all $n\ge 1$ for some constant $c''>0$. Then the desired lower bound in the assertion holds by letting $c_{1}=\min(c',c'')$. }
\end{proof}

\begin{proof}[Proof of Theorem~\ref{thm:EV_LB}]
	Let $\epsilon=1-2p$.
	Given $\epsilon$, we define 
	\begin{equation}\label{eq: r_choice}
	r = r(\epsilon) = \commHL{\left\lfloor \frac{1}{2} (c_{1}/\eps)^{-2/d} \right\rfloor,}
	\end{equation}
	where $c_1$ is the constant in \thref{lem:devr}. Thus, this lemma shows
  that if $X$ is the number of $A$-particles that survive forever
  for DLAS on $\TT_r^d$, then
	$\mathbb{E}_pX \geq c_1^2(2r)^{d/2}$.
	Let $\bar\rho_t$ be the density of $A$-particles in DLAS on $\TT_{r}^d$. By translation invariance, for all $s\ge 0$ and $p\in [1/4,1/2)$,
	\begin{align}
	\bar\rho_s \geq \frac{c_1^2(2r)^{d/2}}{\abs{\TT_r^d}} = c_{1}^{2} (2r)^{-d/2} \geq 2c_1^*\epsilon,
	\label{eq:barrho_s}
	\end{align}
	where $c_1^*$ is a constant depending only on $d$.
	
  Now we set $t=t(\eps)=c_2r(\eps)^2/\log r(\eps)$ 
  and apply \thref{lem:torus} to obtain
  \begin{align*}
    \abs{\rho_t-\bar\rho_t} \leq c_1^*\epsilon,
  \end{align*}
  when $c_2=c_2(d)$ is chosen sufficiently small.
  Applying \eqref{eq:barrho_s}, we deduce that $\rho_t\ge c_{1}^{*}\epsilon$ for all $p\in [1/4,1/2)$. Note that $\rho_s$ is decreasing in $s$ since particles can only disappear (see  \cite[Lemma~2]{vicius_DLA} for a formal proof). Hence we have $\rho_s\ge  c_{1}^{*} \epsilon$ for all $s\leq t$ and $p\in [1/4,1/2)$. Therefore, for all $p\in [1/4,1/2)$,
	\begin{align*}
	\E_p V_\infty \ge \E_{p} V_{t} &=\int_0^t\rho_s\,ds \ge c_{1}^{*}\eps t \ge  C\frac{\epsilon^{-4/d+1}}{-\log \epsilon}
	\end{align*} 
	for some constant $C>0$ that may only depend on $d$.
\end{proof}

\subsection{An upper bound on \texorpdfstring{$\E V_t$}{E V\_t} for DLAS on \texorpdfstring{$\mathbb{Z}$}{Z}} \label{sec:UB}
 
In this subsection, we prove \thref{thm:UB}, showing that occupation time of the origin
in DLAS on $\ZZ$ is $O(t^{3/4})$ when $p=1/2$ and $\lambda_B=0$. 

We will work with the sequential process
defined in Section~\ref{sec:lemmas}. Recall that in this process,
we order the graph's vertices and then run the $A$-particles one at a time in sequence.
Each $A$-particle runs until
annihilation or time~$t$. By \thref{lem:sequential.process}, the total occupation
time of the origin in this process is stochastically larger than the total occupation time
of the origin in  DLAS up to time~$t$. Until the final proof
of \thref{thm:UB}, we will work with the sequential process with
particles present initially only on the positive integers, with vertices ordered $1,2,\ldots$.
We call this the \emph{one-sided sequential process}.

We start by proving an estimate on the probability of an $A$-particle at $k$ reaching the origin. We then use this to bound the expected occupation time of the origin in the one-sided sequential process.

  \begin{lemma}\thlabel{lem:seq}
    \commHL{Assume $p=1/2$.}    Let $G_k$ be the event
    that there is an $A$-particle initially at $k$ and that it visits the root
    in the one-sided sequential process run for time~$t$.
    Then for $1\leq k\leq 2t$,
    \begin{align*}
      \P(G_k)\leq Ck^{-1/2}e^{-k^2/12t}
    \end{align*}
    for an absolute constant $C$.
  \end{lemma}
  \begin{proof}
    We start by defining two random variables that are functions of the initial
    particle configuration. Fix $k$, and
    let $D=D[1,k-1]$ be the number of $A$-particles minus the number of $B$-particles initially
    in $[1,k-1]$. Observe that $D$ is distributed as
    $2\Bin(k-1,1/2)-(k-1)$. Consider the $i$th $B$-particle to the right of position~$k$, and
    let $L_i$ be its distance from $k$. Since $L_i$ is a sum of $i$ independent random variables
    with the geometric
    distribution of parameter $1/2$ on $\{1,2,\ldots\}$,
    \begin{align}
      \E L_i=2i. \label{eq:E[L_i]}
    \end{align}

    Let $\Ff_{k-1}$ be the $\sigma$-algebra generated by the locations of all $B$-particles
    and the paths of $A$-particles starting at positions $1,\ldots,k-1$ in the one-sided
    sequential process. This represents the information available after running
    the process for the $A$-particles at positions $1,\ldots,k-1$.
    Let $F_k$ be the event that no $B$-particles remain at these sites and that
    there is an $A$-particle at position~$k$.
    The event $G_k$ cannot occur unless
    $F_k$ occurs, since an $A$-particle at site~$k$ cannot move to the origin
    without colliding with one of these $B$-particles.
    If $F_k$ occurs,
    then the key quantity is the distance to the right of $k$ that the first remaining
    $B$-particle is found. Call this value $R$, setting it to $0$ if $F_k$ does not occur. 
    Given $R$ and that $F_k$
    occurs, the question is whether a simple random
    walk will reach the origin from $k$ in $t$ steps without moving to site $R+k$.
    Hence, using the gambler's ruin probability of reaching the origin without
    moving to site $R+k$ in any number of steps, we obtain the bound
    \begin{align*}
      \P(G_k\mid \Ff_{k-1})\leq \frac{R}{R+k}\leq\frac{R}{k}
    \end{align*}
    for all $k\geq 1$. When $3\sqrt{t}\leq k \leq 2t$,
    we can apply \thref{lem:gamblers.ruin} to obtain
    \begin{align*}
      \P(G_k\mid \Ff_{k-1})\leq \frac{2R}{R+k}e^{-k^2/12t}\leq\frac{2R}{k}e^{-k^2/12t}.
    \end{align*}
    Since $e^{-k^2/12t}$ is bounded away from $0$ for $1\leq k < 3\sqrt{t}$,
    it holds for all $1\leq k\leq 2t$ that
    \begin{align}\label{eq:gr1}
      \P(G_k\mid \Ff_{k-1}) = O\biggl(\frac{R}{k} e^{-k^2/12t}\biggr).
    \end{align}
    
    We claim that $R\leq L_{1+D}$ on the event $F_k$. Indeed, for $F_k$ to occur,
    the $A$-particles initially in $[1,k-1]$ must annihilate all $B$-particles on that interval
    (and in particular $D\geq 0$).
    The remaining $D$ $A$-particles can then annihilate at most the first $D$ $B$-particles
    to the right of position~$k$, which makes $R$ at most the distance from $k$ to
    the $(1+D)$th $B$-particle to its right.
    And if $F_k$ does not occur, then $R=0$. Hence $R\leq L_{1+\abs{D}}$ regardless of whether
    $F_k$ occurs. Thus it follows
    from \eqref{eq:gr1} that
    \begin{align}
      \P\bigl(G_k\bigmid \Ff_{k-1}\bigr) = O\biggl(\frac{L_{1+\abs{D}}}{k}e^{-k^2/12t}\biggr). \label{eq:G|L}
    \end{align}
    Hence 
    \begin{align}
      \P(G_k\mid D)= \E\biggl[ \P\bigl(G_k\bigmid \Ff_{k-1}\bigr)\biggmid D\biggr]&\leq
        \E\Biggl[ \frac{L_{1+\abs{D}}}{k}\Biggmid D\Biggr]O\bigl(e^{-k^2/12t}\bigr)\\
        &=\frac{2(1+\abs{D})}{k}O\bigl(e^{-k^2/12t}\bigr), \label{eq:G|D}
    \end{align}
    where the final equality uses \eqref{eq:E[L_i]} together with the independence
    of $D$ from $(L_i)_{i\geq 1}$.
    Taking expectations gives
    \begin{align*}
      \P(G_k) \leq \frac{1}{k}(1 + \E\abs{D})O\bigl(e^{-k^2/12t}\bigr)&\leq 
      \frac{1}{k}\bigl(1 + \sqrt{\E D^2}\bigr)
      O\bigl(e^{-k^2/12t}\bigr)\\
        &= O(k^{-1/2}e^{-k^2/12t}\bigr).
    \end{align*}   
    This finishes the proof.
  \end{proof}

\newcommand{\seqocc}{U^+_t}
\begin{proposition}\thlabel{prop:one.sided.bound}
  Let $\seqocc$ be the total occupation time of the origin
  in the one-sided sequential process run for time~$t$. Then $\E \seqocc = O\bigl(t^{3/4}\bigr)$.
\end{proposition}
\begin{proof}
  Let $X_k$ be the total time that an $A$-particle starting at position~$k$ occupies
  the origin in the one-sided sequential
  process run for time~$t$.
  As in \thref{lem:seq}, let $G_k$ be the event that there is an $A$-particle initially at position~$k$
  that visits the origin. Then
  \begin{align*}
    \E X_k = \E[X_k\mid G_k]\P(G_k)\leq t^{1/2}\P(G_k),
  \end{align*}
  bounding $\E[X_k\mid G_k]$ by the occupation time of the origin up to time~$t$
  by a random walk starting at the origin, which is at most $t^{1/2}$ by
  \thref{lem:local_time}.
  
  We break $\E \seqocc$ into two parts:
  \begin{align}\label{eq:Eseqocc}
    \E \seqocc &= \sum_{k=1}^{\floor{2t}}\E X_k + \sum_{k=\floor{2t}+1}^{\infty} \E X_k
      \leq t^{1/2}\Biggl(\sum_{k=1}^{\floor{2t}}\P(G_k) + \sum_{k=\floor{2t}+1}^{\infty} \P(G_k)\Biggr).
  \end{align}
  For the first sum, we apply \thref{lem:seq} and then bound the sum by an integral
  to get
  \begin{align}\label{eq:sum1}
    \begin{split}
    \sum_{k=1}^{\floor{2t}}\P(G_k) \leq C\int_0^\infty x^{-1/2}e^{-x^2/12t}\,dx
     &= C\int_0^\infty t^{1/4}u^{-1/2}e^{-u^2/12}\,du\\&=O(t^{1/4}).
    \end{split}
  \end{align}
  For the second sum, we bound $\P(G_{\floor{2t}+1+i})$ by the probability
  of a random walk having displacement $2t+i$ in time~$t$ and
  apply \thref{lem:poisson_tail} to get
  \begin{align*}
    \P(G_{\floor{2t}+i+1})\leq \exp\Biggl(-\frac{(t+i)^2}{2(2t+i)}\Biggr).
  \end{align*}
  Hence
  \begin{align}\label{eq:sum2}
    \sum_{k=\floor{2t}+1}^{\infty} \P(G_k) &\leq \sum_{i=0}^{\infty}\exp\Biggl(-\frac{(t+i)^2}{2(2t+i)}\Biggr) \\
    &\le  \sum_{i=0}^{\infty}\exp\Biggl(-\frac{(t+i)}{4}\Biggr) 
    \le \int_{t-1}^{\infty} \exp(-x/4)\,dx = O\bigl(e^{-t/4}\bigr).
  \end{align}
  Equations \eqref{eq:sum1} and \eqref{eq:sum2} together with \eqref{eq:Eseqocc}
  prove the proposition.
\end{proof}

\begin{remark}
  It is possible to avoid the work of proving \thref{lem:gamblers.ruin} in the appendix
  as follows. First,
  use the usual gambler's ruin computation rather than \thref{lem:gamblers.ruin} in \thref{lem:seq}, proving  
  only that $\P(G_k)=O(k^{-1/2})$. Then in
  \eqref{eq:Eseqocc}, break $\E\seqocc$ into three sums, bounding the first
  using the estimate $\P(G_k)=O(k^{-1/2})$, the second using the moderate deviations
  estimate \thref{lem:SRW} for the random walk, and the last using
  \thref{lem:poisson_tail} as was done. The downside of this approach is that it adds
  an extra logarithmic factor to the bound given in \thref{prop:one.sided.bound}.
\end{remark}

Now we have all of the necessary estimates to prove our theorem. 
\begin{proof}[Proof of Theorem~\ref{thm:UB}]
  Let $V^+_t$ and $V^-_t$ denote the total occupation time 
  of the origin by $A$-particles in DLAS on $\ZZ$ with $p=1/2$
  up to time~$t$ 
  with particles placed initially only at positive integers
  and only at negative integers, respectively. Let $V^0_t$ denote the occupation time of
  the origin by $A$-particles in
  DLAS with only a single particle started at the origin (i.e., the local time of the origin
  by a single random walk if an $A$-particle is placed at the origin, and zero otherwise).
  Applying \thref{lem:polarized_DLAS} twice, we have 
  \begin{align*}
    \E V_t\leq \E V^+_t+\E V^0_t+\E V^-_t.
  \end{align*}
  By \thref{lem:sequential.process,prop:one.sided.bound}, we have 
  $\E V^+_t\leq \E\seqocc=O(t^{3/4})$, and by symmetry the same bound holds
  for $\E V^-_t$. Finally $\E V^0_t\leq t^{1/2}$ by \thref{lem:local_time}, completing the proof that
  $\E V_t=O(t^{3/4})$.
\end{proof}

\subsection{An upper bound on the critical exponent for DLAS on \texorpdfstring{$\mathbb{Z}$}{Z}} \label{sec:EV_UB}

In this section, we prove \thref{thm:EV_UB}.
As in our proof of \thref{thm:UB}, we will use the sequential version of DLAS defined in 
Section~\ref{sec:lemmas}. Consider the sequential process run for infinite time 
with all particles in the initial configuration removed from the negative integers.
Order the vertices $0,1,2,\ldots$. Let
$U_k$ be $0$ if a $B$-particle starts at site~$k$; otherwise,
let it be the total time spent at the root by the $A$-particle initially
at position~$k$. If $k=0$ and there is an $A$-particle at $0$, we let $U_0$ be the number of visits to $0$ by that particle for $t \geq 1$.
The following estimate is most of the work in proving \thref{thm:EV_UB}.

\begin{lemma}\thlabel{lem:Uk}
For $k\ge 1$, let $D = D[0,k-1]$. Then
\begin{align*}
    \E U_0 = \ind{\text{$0$ contains an $A$-particle}}+ 2p(1-p)^{-1} \ind{\text{$1$ contains an $A$-particle}}
  \end{align*}
  and
  \begin{align*}
    \E U_k = 2p(1-p)^{-1}\E\bigl[ \ind{D \geq 0}(1 +D) \bigr].
  \end{align*}
\end{lemma}
\begin{proof}\newcommand{\bG}{k\to\root}\newcommand{\G}{\{\bG\}}%
  We first describe the case $k \geq 1$. As in the proof of \thref{lem:seq}, let $L_i$ be the distance from position~$k$ to
  the $i$th $B$-particle to its right.
   Let $L = \ind{D \geq 0} L_{D+1}$. Let $\G$ denote the event that there is an $A$-particle at $k$ and it reaches the origin for some $t \geq 1$.
   In the sequential process we consider, each $A$-particle in $[0,k-1]$ is eventually annihilated,
   since there are infinitely many $B$-particles in the initial configuration. Thus, if $D\geq 0$
   and there is initially an $A$-particle at $k$, 
   then the $A$-particles initially present in $[0,k-1]$ will annihilate all $B$-particles
   in that region as well as the first $D$ $B$-particles to the right of $k$. 
   Hence, when the $A$-particle at $k$ is run, there are no $B$-particles to its
   left, and the closest $B$-particle to its right is at distance $L_{D+1}$.
   By the gambler's ruin calculation,
$$\P(\bG \mid L) = \f{pL}{k + L}.$$
Conditional on $L$, each time the particle visits the origin, it is annihilated without returning with probability 
$$r = \f 12 \f{ 1}{k + L}.$$
This is the probability it moves to the right on its first step and then reaches $L$ before returning to to the origin. 
Hence, conditional on $\G$, the number of visits to $\root$ by the particle at $k$ is distributed
geometrically on the positive integers with success probability $r$.
Since the particle stays at $\root$ for expected time~$1$ on each visit,
\begin{align}\label{eq:Uk|L}
  \E[ U_k\mid L] = \frac{\P(\bG\mid L)}{r} = 2pL.
\end{align}
As in \eqref{eq:E[L_i]}, we have $\E L_i=i/(1-p)$, since $L_i$ is a sum of $i$ independent
geometric random variables with success probability $1-p$. Since $L=\1\{D\geq 0\}L_{1+D}$,
\begin{align}\label{eq:L|D}
  \E [ L  \mid D ] = (1-p)^{-1}\ind{D \geq 0} (1+D).
\end{align}
Taking expectations in \eqref{eq:Uk|L} and \eqref{eq:L|D} completes the proof for $k \geq 1$.

When $k=0$ and there is an $A$-particle at $0$, the situation is slightly different.  The reason is that the particle will move to the negative integers and then eventually back to $0$ a Geometric distributed with mean 1 number of times before moving to $1$ for the first time. After doing so, if $1$ contains a $B$-particle, then the $A$-particle is destroyed. If $1$ contains an $A$-particle, then we may repeat the argument for $U_1$ with $D=0$. This gives the claimed formula. 
\end{proof}


\begin{proof}[Proof of Theorem \ref{thm:EV_UB}]
  Let $U^+$ be the occupation time of $\root$ in
 the sequential process on the halfline run for infinite time, and observe that $U^+=\sum_{k=1}^{\infty}U_k$.
 It suffices to show that $\E_p U^+\leq C (1-2p)^{-3}$ for an absolute
 constant $C$, since then Lemmas~\ref{lem:sequential.process} and \ref{lem:polarized_DLAS} complete the proof
 as in the conclusion of the proof of \thref{thm:UB}.
 
 Let $(S_k)_{k=0}^{\infty}$ be a \commHL{simple} random walk on the integers jumping a step in the positive
 direction with probability~$p$ and in the negative with probability~$1-p$.
 By \thref{lem:Uk},
 \begin{align*}
   \E U^+ = \sum_{k=1}^\infty \E U_k = \frac{2p}{1-p}\sum_{k=1}^{\infty}\E\bigl[ \ind{S_k \geq 0}(1 +S_k) \bigr].
 \end{align*}
 This sum is essentially the expected area under the positive excursions of a random walk
 with negative drift, which we can compute exactly. We rewrite the sum to get
 \begin{align*}
   \E U^+
     = \frac{2p}{1-p}\sum_{u=0}^{\infty}\E\Biggl[\sum_{k=1}^{\infty}\ind{S_k\geq u}\Biggr]
     &= \frac{2p}{1-p}\Biggl(\sum_{u=0}^{\infty}\E\Biggl[\sum_{k=0}^{\infty}\ind{S_k\geq u}\Biggr]-1\Biggr).
 \end{align*}
 Given that $(S_k)$ ever hits $u$, the distribution of $\sum_{k=0}^{\infty}\ind{S_k\geq u}$
 is the same as the unconditional distribution of $\sum_{k=0}^{\infty}\ind{S_k\geq 0}$.
 Hence, with $M$ denoting the maximum value taken by $(S_k)$,
 \begin{align}\label{eq:two.sums}
   \E U^+ = \frac{2p}{1-p}\Biggl(\sum_{u=0}^{\infty}\P(M\geq u)\sum_{k=0}^{\infty}\P(S_k\geq 0)-1\Biggr).
 \end{align}
 
 Now we compute the two sums.
 By the gambler's ruin calculation for biased random walk,
 \begin{align}\label{eq:first.sum}
   \sum_{u=0}^{\infty}\P(M\geq u) &= \sum_{u=0}^{\infty}\biggl(\frac{p}{1-p}\biggr)^u = \frac{1-p}{1-2p}.
 \end{align}
 To approach the second sum, let
 $S$ denote the number of steps for a \commHL{simple random walk on the integers} with bias $p$ starting at $1$
 to hit $0$. Since $S$ is equal to $1$ with probability~\commHL{$1-p$} and otherwise is distributed as the sum
 of two independent copies of itself, it satisfies $\E S = 1-p+2p\E S$ and hence $\E S = (1-p)/(1-2p)$.
 Now, let $T=\sum_{k=0}^{\infty}\ind{S_k\geq 0}$. We claim that
 \begin{align}\label{eq:T.recurrence}
   T \eqd \begin{cases}
     1+S+T &\text{with probability $p$,}\\
     1+T & \text{with probability $p$,}\\
     1 & \text{with probability $1-2p$.}
   \end{cases}
 \end{align}
 The first case corresponds to the event that $(S_k)$ initially jumps to the right, which occurs
 with probability~$p$. Then $1+S$ is contributed to the sum before it returns to $0$,
 and the sum from that time on is distributed as $T$ by the strong Markov property.
 In the second case, the first step of $(S_k)$ is to the left \commHL{(with probability $1-p$), and conditional on that event, it eventually returns to $0$ (with probability $p/(1-p)$ by Gambler's ruin) and the sum for $T$ resets. Hence $T$ equals to $1+T$ with probability $p$.} And in the final case, $(S_k)$ never returns to $0$.
 Taking expectations of both sides of \eqref{eq:T.recurrence} and solving the resulting equation yields
 \begin{align}\label{eq:second.sum}
   \E T = \frac{p(3-5p)}{(1-2p)^2}+1.
 \end{align}
 From \eqref{eq:two.sums}, \eqref{eq:first.sum}, and \eqref{eq:second.sum},
 \begin{align*}
   \E U^+ &= \frac{2p}{1-p}\Biggl( \biggl(\frac{1-p}{1-2p}\biggr)\biggl(\frac{p(3-5p)}{(1-2p)^2}+1\biggr)-1\Biggr) = \frac{2p^2(2-3p)^2}{(1-p)(1-2p)^3} = O\bigl((1-2p)^{-3}\bigr).
 \end{align*}
\end{proof}

\section{Mean field behavior on bidirected trees} \label{sec:tree}
Consider the infinite $2d$-regular tree with some vertex $\root$ distinguished as the root.
At each vertex~$v$, orient $d$ of the edges to point toward $v$ and $d$ to point away. Define the
random walk kernel $K$ to send a walker at $v$ along each of the $d$ out-edges with probability
$1/d$. We denote this oriented tree by $\BT$.
In this section, we consider the two-type DLAS on $\BT$ with $\lambda_{B}=0$. As usual,
the initial configuration is one particle per site, where each particle
is independently given type~$A$ with probability~$p$.
The subset of $\BT$ made up of vertices with a path to the root forms
a rooted $d$-ary tree (see Figure~\ref{fig:tree}).
We will ignore the rest of the tree, since no particles from it can contribute to $V_t$.
We define level~$k$ of the tree as the $d^k$ vertices in this subtree 
\commHL{from which there is a directed path of $k$ edges
toward} the root.

\begin{figure}
	\centering
	\begin{tikzpicture}[tv/.style={circle,fill,inner sep=0,
		minimum size=0.15cm,draw}]
	\path (0,0) node[tv,label=right:$\root$] (v) {}
	+(1, 1) node[tv,label=right:$v_2$] (v0) {}
	+(-1, 1) node[tv,label=right:$v_1$] (v1) {};
	\path (v0)+(-.5,1) node[tv] (v00) {}
	+(.5,1) node[tv] (v01) {}
	(v1)+(-.5,1) node[tv] (v10) {}
	+(.5,1) node[tv] (v11) {}
	;
	\draw[-Stealth] (v11)--(v1);
	\draw[-Stealth] (v10)--(v1);
	\draw[-Stealth] (v1)--+(170:.5);
	\draw[-Stealth] (v01)--(v0);
	\draw[-Stealth] (v00)--(v0);
	\draw[-Stealth] (v0)--+(170:.5);
	\draw[-Stealth] (v1)--(v);
	\draw[-Stealth] (v0)--(v);
	\draw[-Stealth] (v)--+(170:.5);
	\draw[-Stealth] (v)--+(260:.5);
	\draw[Stealth-] (v11)--+(110:1);
	\draw[Stealth-] (v11)--+(70:1);
	\draw[-Stealth] (v11)--+(170:.5);
	\draw[Stealth-] (v10)--+(110:1);
	\draw[Stealth-] (v10)--+(70:1);
	\draw[-Stealth] (v10)--+(170:.5);
	\draw[Stealth-] (v01)--+(110:1);
	\draw[Stealth-] (v01)--+(70:1);
	\draw[-Stealth] (v01)--+(170:.5);
	\draw[Stealth-] (v00)--+(110:1);
	\draw[Stealth-] (v00)--+(70:1);
	\draw[-Stealth] (v00)--+(170:.5);
	
	\end{tikzpicture}
	\caption{The tree $\vec{\mathcal{T}}_4$, a $4$-regular tree with oriented edges. Pictured is the
		the portion of the tree
		that leads to the root $\root$, which forms a binary tree.}\label{fig:tree}
\end{figure}

By \thref{lem:finite.range.truncation}, when bounding $\E V_t$
we can ignore particles far from the root.
Thus, our strategy will be to work with DLAS with particles removed beyond some
level~$n$ (which we will later take to be $\ceil{ct}$) and to ignore time,
counting the total number of visits to the root by $A$-particles
in any amount of time. 
The main goal of this section is to prove the following bounds on this quantity.
\begin{proposition}\thlabel{prop:tree.discrete.bound}
  Consider DLAS on $\BT$ with particles initially placed only at levels $1,\ldots,n$.
  Let $W_n$ denote the number of $A$-particles that visit the root in any amount of time.
  \begin{enumerate}[label = (\alph*)]
    \item If $p=1/2$, then for absolute constants $c$ and $C$, it holds for all $d,n\geq 2$ that
      \begin{align*}
        c\log n\leq \E W_n \leq C\log n.
      \end{align*}\label{i:tdb.critical}
    \item If $\epsilon=1/2-p>0$, then $\E W_n$ approaches a finite limit as $n\to\infty$,
      and for absolute constants $c$, $C$, and $\eta$, it holds for
      all $0<\epsilon<\eta$ and all $d\geq 2$ that
      \begin{align*}
        c\log\biggl(\frac{1}{\epsilon}\biggr)\leq \lim_{n\to\infty}\E W_n \leq C\log\biggl(\frac{1}{\epsilon}\biggr).
      \end{align*}
      \label{i:tdb.subcritical}
  \end{enumerate}
\end{proposition}

Before we go further, we show how these bounds prove Theorems~\ref{thm:tree} and \ref{thm:tree.exponent}.
First, we invoke \thref{lem:monotonicity,lem:finite.range.truncation} to relate $V_t$ and $W_n$.
\begin{lemma}\label{lem:VW.relationship}
  For some absolute constant $c$,
  \begin{align*}
    \E W_{\floor{t/2}} - O(1) \leq \E V_t \leq \E W_{\ceil{ct}} + O(1).
  \end{align*}
\end{lemma}
\newcommand{\truncV}[1]{\overline{V}_{#1}}
\begin{proof}
  We start with the upper bound. Let $c$ be the constant from \thref{lem:finite.range.truncation}.
  Fix $t\ge 0$, and for each $s\ge 0$, let $\overline{\rho}_s$ and $\truncV{s}$ respectively denote the density
  and the total occupation time at time $s$ of the root by $A$-particles
  in DLAS on $\BT$ with particles beyond level $ct$ removed from the initial configuration.
  Then by \thref{lem:finite.range.truncation},
  \begin{align*}
    \E V_t \leq \E \truncV{t} + \int_0^t\abs{\rho_s-\overline{\rho}_s}\,ds \leq \E \truncV{\infty}+te^{-t}.
  \end{align*}
  Since any $A$-particle visiting the root is either immediately annihilated by a $B$-particle
  or remains there for expected time $1$, we have
  $\E \truncV{\infty}\leq \E W_{\ceil{ct}}+p$ 
  (the reason for the extra $p$ is that $\truncV{\infty}$ includes the possible contribution that
  an $A$-particle initially at the root provides to the root's occupation time.)
  This completes the proof of the upper bound.
  
  For the lower bound, let $\tilde{V}_{t}$ denote the total occupation time of the root
  by $A$-particles 
  in DLAS on $\BT$ with the starting configuration containing $A$-particles only at levels
  $1,\ldots,\floor{t/2}$.
  By \thref{lem:monotonicity}, we have $\E V_{t}\geq\E \tilde{V}_{t}$. Let $\tilde{N}_{t}$ be the total number
  of $A$-particles in this truncated system that visit the root in time~$t$.
  Since each $A$-particle moving to the root stays there
  for expected time~$1$ except possibly the first one (which is immediately annihilated
  if the root initially contains a $B$-particle), we have $\E\tilde{V}_{t}\geq \E \tilde{N}_{t}-1$.
  Thus it suffices to show that $\E \tilde{N}_{t}\geq \E W_{\floor{t/2}}-O(1)$.
  
  Let $\tilde{N}$ be the number of $A$-particles in this truncated initial configuration 
  whose underlying random walk paths
  visit the root, whether or not they are annihilated before reaching it. There
  are $d^k$ vertices at level~$k$, and a random walk starting at such a vertex
  has probability $1/d$ of moving toward the root at each step and hence has probability $d^{-k}$
  of ever visiting the root. By expressing $\tilde{N}$ as a sum of indicators, we find that
  $\E \tilde{N} =O(t)$ and $\E \tilde{N}^2 = O(t^2)$. Now, let $E$ be the event that the underlying random walk
  for each particle counted by $\tilde{N}$ makes at least $t/2$ jumps in time $t$. On this event,
  we have $\tilde{V}_t = W_{\floor{t/2}}$, since all particles that ever visit the origin
  will do so in time $t$. It follows that 
  \begin{align*}
    \tilde{V}_t \geq W_{\floor{t/2}}\1_E \geq W_{\floor{t/2}} - \tilde{N}\1_{E^c}.
  \end{align*}
  
  Now we bound $\E[\tilde{N}\1_{E^c}]$. The probability that a random walk path jumps fewer than
  $t/2$ times in time $t$ is bounded by $e^{-bt}$ for some $b>0$, by basic
concentration  properties of the Poisson distribution. Thus $\P(E^c\mid \tilde{N})\leq \tilde{N}e^{-bt}$.
  This yields $\E[\tilde{N}\1_{E^c}]\leq e^{-bt}\E \tilde{N}^2=O(1)$, and therefore
  $\E \tilde{V}_t\geq \E W_{\floor{t/2}} - O(1)$.
\end{proof}

\begin{proof}[Proof of Theorem~\ref{thm:tree}]
  The theorem follows immediately by applying Lemma~\ref{lem:VW.relationship}
  followed by \thref{prop:tree.discrete.bound}\ref{i:tdb.critical}.
\end{proof}

\begin{proof}[Proof of Theorem~\ref{thm:tree.exponent}]
  Let $\epsilon=1-2p$.
  By Lemma~\ref{lem:VW.relationship} and \thref{prop:tree.discrete.bound}\ref{i:tdb.subcritical},
  \begin{align*}
    \limsup_{t\to\infty} \E_p V_t &\leq C\log(1/\epsilon) + O(1)\leq C'\log(1/\epsilon)\\
    \intertext{and}
    \liminf_{t\to\infty} \E_p V_t &\geq c\log(1/\epsilon) - O(1) \geq c'\log(1/\epsilon)
  \end{align*}
  for sufficiently small $\epsilon$.
  Since $V_t$ converges upwards to $V_\infty$, by the monotone convergence theorem
  $\E_p V_{\infty}$ exists and is bounded above by $C'\log(1/\epsilon)$ and below by $c'\log(1/\epsilon)$.
\end{proof}

\subsection{A recursive distributional equation for \texorpdfstring{$W_n$}{W\_n}}\label{sec:rde}

We now set out to prove \thref{prop:tree.discrete.bound}.
The main idea to understand $W_n$
is to take advantage of the recursive structure of the tree.
Let $v_1,\ldots,v_d$ be the vertices at level~$1$ of the tree (see Figure~\ref{fig:tree}).
For any of these vertices, the number of $A$-particles visiting them in this system
is distributed as $W_{n-1}$. Because the tree is directed, these counts are independent.
Thus, we obtain a distributional equation expressing
the distribution of $W_n$ in terms
of $d$ independent copies of $W_{n-1}$ as follows. Let $W_{n-1}^{(1)},\ldots,W_{n-1}^{(d)}$
denote these counts. Let $Y_i$ be $1$ if $v_i$ initially contains an $A$-particle and be $-1$
if it initially contains a $B$-particle. From the dynamics of the system, the number of
$A$-particles with a chance of moving from $v_i$ to the root is $\bigl(W_{n-1}^{(i)}+Y_i\bigr)^+$,
where $(\cdot)^+ = \max\{0,\cdot\}$. Each of these $A$-particles independently moves
to the root with probability $d^{-1}$. Hence,
\begin{align}\label{eq:W_n.Aa}
  W_n \eqd \sum_{i=1}^d \Bin\Bigl(\bigl(W_{n-1}^{(i)} + Y_i\bigr)^+,\,d^{-1}\Bigr).
\end{align}
Here we use the notation $\Bin(X,p)$ to denote the $p$-thinning of a nonnegative integer--valued
random variable $X$. That is, $\Bin(X,p)$ is defined as $\sum_{i=1}^X B_i$, where
$B_1,B_2,\ldots$ are independent of each other and $X$ and have distribution $\Ber(p)$.

Notationally, we will express this relationship between $W_{n-1}$ and $W_n$ by defining
an operator $\Aa$ on probability distributions that maps the law of $W_{n-1}$ to the law of $W_n$.
Let $Y_1,\ldots,Y_d$ be independent, each equal to $1$ with probability~$p$ and $-1$ with probability $1-p$.
Let $X_1,\ldots,X_d$ be independent copies of an arbitrary random variable $X$ taking
values in the nonnegative integers.
Then we define the result of applying $\Aa$ to the law of $X$ by
\begin{align}\label{eq:Adef}
  \Aa X = \sum_{i=1}^d \Bin\Bigl(\bigl(X_i + Y_i\bigr)^+,\,d^{-1}\Bigr).
\end{align}
We will sometimes abuse notation and treat $\Aa X$ as a random variable with this distribution.
The following lemma summarizes the fact given in \eqref{eq:W_n.Aa}:
\begin{lemma}\thlabel{lem:AW}
  The random variables $(W_n)_{n\geq 0}$ from \thref{prop:tree.discrete.bound} satisfy
  the distributional equality
  \begin{align*}
    W_{n+1}\eqd\Aa W_n.
  \end{align*}
\end{lemma}

\commTJ{This operator is monotone with respect to the stochastic order.}
\begin{lemma}\thlabel{lem:A.monotone}
  If $X\preceq X'$, then $\Aa X\preceq \Aa X'$.
\end{lemma}
\begin{proof}
  Let 
  \begin{align*}
    \Aa X &= \sum_{i=1}^d \Bin\Bigl(\bigl(X_i + Y_i\bigr)^+,\,d^{-1}\Bigr)\\\intertext{and}
    \Aa X' &= \sum_{i=1}^d \Bin\Bigl(\bigl(X'_i + Y_i\bigr)^+,\,d^{-1}\Bigr),
  \end{align*}
  with $X_1,\ldots,X_d$ i.i.d.\ copies of $X$ and $X'_1,\ldots,X'_d$ i.i.d.\ copies of $X'$.
  Since $X\preceq X'$, we can couple each $X_i$ with $X_i'$ so that $X_i\leq X_i'$ a.s.,
  thus producing a coupling in which $\Aa X\leq \Aa X'$ a.s.
\end{proof}

Our goal now is to use the operator $\Aa$ to analyze the growth of $\E W_n$. 
We make the following observation:    
\begin{lemma}\thlabel{lem:A.growth}
  Let $X$ be a random variable taking nonnegative integer values.
  Then
  \begin{align}\label{eq:A.growth}
   \commHL{ \E[\Aa X] } = \E X + 2p-1 + (1-p)\P(X=0).
  \end{align}
\end{lemma}

\begin{proof}
  This follows from the definition of $\Aa$ and the statement
  \begin{align*}
    \E \bigl[ (X_i+Y_i)^+ \bigr] &= \E\bigl[X_i+Y_i +\ind{X_i=0,\,Y_i=-1}\bigr]=
     \E X + \E Y_i + \P(X=0)\P(Y_i=-1),
  \end{align*}
  where $X_1,\ldots,X_d$ and $Y_1,\ldots,Y_d$ are independent with $X_i\eqd X$
  and $Y_i$ equal to $1$ with probability $p$ and to $-1$ with probability $1-p$.
\end{proof}

This shows that the growth of $W_n$ depends on its concentration, in that
$\P(W_n=0)$ is the key quantity in \eqref{eq:A.growth} when we apply it
to bound $\E W_{n+1}-\E W_n$. Heuristically, the explanation for the growth rate of $\E W_n$
is that $W_n$ has a Poisson-like lower tail, and in particular
$\P(W_n=0)\approx e^{-c\E W_n}$.
In the critical case $p=1/2$, \thref{lem:A.growth} then gives
\begin{align*}
  \E W_{n+1}-\E W_n \approx C e^{-c\E W_n}.
\end{align*}
Solving this difference equation shows that $\E W_n$ grows logarithmically.
In the subcritical case, a similar difference equation shows that
$\E W_n$ converges as $n$ tends to infinity and gives bounds on the limit.

The lower bounds in \thref{prop:tree.discrete.bound} are proven just as in this description.
We need an anticoncentration bound of the form $\P(W_n=0)\geq Ce^{-c\E W_n}$.
To get this, we consider the worst-case scenario, when $W_n$ is as concentrated as possible.
One can then calculate that even in this case, $\Aa W_n$ is not too concentrated,
and therefore $W_{n+1}$ satisfies the desired anticoncentration bound.

The upper bounds in \thref{prop:tree.discrete.bound} are more difficult. For technical reasons,
we bound the expectations of a sequence of random variables $(U_n)_{n\geq 0}$
that serves as a stochastic upper bound for $(W_n)_{n\geq 0}$. We then show an exponential
concentration bound for the lower tail of $U_n$ using the theory of size-bias couplings.

In Sections~\ref{sec:size.bias} and \ref{sec:log.concavity},
we lay out some background material we will need for this argument. The upper bounds
from Proposition \ref{prop:tree.discrete.bound} are proven in Section~\ref{sec:tree.upper}
followed by the lower bounds in Section~\ref{tree.lower}.

\subsection{Concentration by size-bias coupling}\label{sec:size.bias}
  Let $\pi$ be a probability distribution on the nonnegative real numbers
  with mean $\mu$.
  The \emph{size-bias transform} of $\pi$ is the distribution $\pi^s$ with
  Radon--Nikodym derivative
  \begin{align*}
    \frac{d\pi^s(x)}{d\pi} = \frac{x}{\mu}.
  \end{align*}
  We will often abuse notation and speak of the size-bias transform of a random
  variable as another random variable, rather than referring to their distributions.
  In this spirit, the size-bias transform of a random variable $X$ on the nonnegative integers
  is the random variable $X^s$ with distribution given by
  \begin{align*}
    \P(X^s=k) = \frac{k}{\E X}\P(X=k).
  \end{align*}
  We will in general use the notation $X^s$ to denote the size-bias transform of $X$.
  
  The size-bias transform comes up in a variety of contexts; see \cite{AGK} for
  a broad survey. Our interest here has its roots in Stein's method for distributional
  approximation. If a random variable and its
  size-bias transform are close to each other in the right sense, one can use Stein's method
  to prove that the random variable is approximately Gaussian or Poisson (see 
  \cite[Sections~3.4 and 4.3]{Ross}).
  It was eventually realized that the same approach could be used to prove
  concentration inequalities \cite{GG,AB,CGJ}. The starting point is
  that if $X$ and its size-bias transform exactly satisfy $X^s\smash{\eqd} X+1$, then $X$
  is Poisson \cite[Corollary~11.3]{AGK}.  If instead $X^s\preceq X+1$,
  then it can be proven that $X$ satisfies a Poisson-like tail concentration bound
  \cite[Theorems~1.2--1.3]{AB}. We use a result that relaxes the condition $X^s\preceq X+1$
  further. A coupling of $X$ and its size-bias transform $X^s$ is
  \emph{$(c,p)$-bounded for the upper tail} if
  \begin{align*}
    \P(X^s\leq X+c\mid X^s)\geq p\text{ a.s.},
  \end{align*}
  and is \emph{$(c,p)$-bounded for the lower tail} if
  \begin{align*}
    \P(X^s\leq X+c\mid X)\geq p\text{ a.s.}
  \end{align*}
  In this language, the condition $X^s\preceq X+1$ is that
  $X$ admits a size-bias coupling that is $(1,1)$-bounded for the upper and lower tails.
  
  \begin{proposition}[{\cite[Theorem~3.3]{CGJ}}]\thlabel{thm:concentration}
  Let $\mu=\E X$ and let $h(x)=(1+x)\log(1+x)-x$ for $x\geq -1$.
  \begin{enumerate}[(a)]
    \item \label{i:concentration.upper}
      If $X$ admits a size-bias coupling that is $(c,p)$-bounded for the 
      upper tail, then for all $x\geq 0$,
      \begin{align*}
        \P\bigl(X-\mu/p\geq x\bigr) \leq \exp\Biggl(-\frac{\mu}{cp}h\Biggl(\frac{px}{\mu}\Biggr)\Biggr)
                          \leq \exp\Biggl(-\frac{x^2}{2c(x/3+\mu/p)}\Biggr).
      \end{align*}
    \item \label{i:concentration.lower}
      If $X$ admits a size-bias coupling that is $(c,p)$-bounded for the 
      lower tail, then
      for all $0\leq x\leq p\mu$,
      \begin{align*}
        \P\bigl(X-p\mu\leq -x\bigr)\leq
        \exp\Biggl(-\frac{p\mu}{c}h\Biggl(-\frac{x}{p\mu}\Biggr)\Biggr) \leq \exp\Biggl( -\frac{ x^2}{2pc\mu}\Biggr). 
      \end{align*}
  \end{enumerate}
  
\end{proposition}
In particular, under the hypotheses of part~(b) of this theorem, we have $\P(X=0)\leq\exp(-p\mu/\commHL{2}c)$
by setting $x=p\mu$. (In \cite{CGJ}, part~(b) of this theorem is stated with the condition $x<p\mu$, 
but taking a limit extends it to $x=p\mu$.)

To use this theorem, we will need to compute size-bias transforms.
In general, this can be done effectively for random variables that are
sums of simpler random variables (e.g., indicators), which need not be independent
(see \cite[Section~2.4]{AGK}).
The case of a size-bias transform of an independent sum is the simplest one:
just choose a random summand with probability in proportion to its expectation, and then apply the transform
to that summand only. The following lemma is a formal statement of this:
\begin{lemma}[{\cite[eq.~(25)]{AGK}}]\thlabel{lem:bias.sum}
  Let $S=X_1+\cdots + X_n$, where $X_1,\ldots,X_n$ are nonnegative \commHL{independent}  random variables.
  Choose $I$ from $\{1,\ldots,n\}$ independently of all else, 
  choosing $I=i$ with probability $\E X_i/\E S$.
  Then
  \begin{align*}
    S^s \eqd X_1 + \cdots + X_n + X_I^s - X_I.
  \end{align*}
\end{lemma}

Next, we compute the size-bias transform of a $p$-thinning:
\begin{lemma}\thlabel{lem:bias.thinning}
  \commTJ{Let $X$ be a nonnegative random variable with finite expectation.}
  Let $Y=\Bin(X,p)$. Then
  \begin{align}\label{eq:bias.thinning}
    Y^s\eqd 1 + \Bin(X^s-1, p).
  \end{align}
\end{lemma}
\begin{proof}
  \commTJ{Rewrote entire proof with more details --TJ}
  To put this statement in a more general context, suppose $Y$ is a mixture
  of random variables $Z_n$ governed by a probability measure $h$; that is,
  for Borel sets $\mathscr{B}$,
  \begin{align*}
    \P(Y\in\mathscr{B}) = \int \P(Z_n\in\mathscr{B})\,h(dn).
  \end{align*}
  Then $Y^s$ is the mixture of the transforms $Z_n^s$, but governed not by $h$
  but by $h^s$, the measure whose Radon--Nikodym derivative satisfies $\frac{dh^s}{dh}(n) = \frac{\E Z_n}{\E Y}$ \cite[Lemma~2.4]{AGK}.
  
  In our case, $Y$ is the mixture of $Z_n\sim\Bin(n,p)$ governed by the law of $X$.
  By \thref{lem:bias.sum}, the size-bias transform of $Z_n$ is $1 + \Bin(n-1, p)$.
  Since $\E Z_n$ is proportional to $n$, the measure $h'$ is the size-bias transform of the law of $X$.
  Thus $Y^s$ is the mixture of $1+\Bin(n-1,p)$ governed by the law of $X^s$, as stated in
  \eqref{eq:bias.thinning}.
\end{proof}

Finally, size-biasing ignores any mass the distribution places on $0$:
\begin{lemma}\thlabel{lem:bias.nonzero}
  Let $X^>$ denote a random variable distributed as $X$ conditioned
on $X>0$. Then $(X^>)^s\eqd X^s$.
\end{lemma}
\begin{proof}
  This follows immediately from the definition of the size-bias transform.
\end{proof}

The size-bias transform does not mesh well with the usual stochastic order,
in that it is not true in general that $X\stprec Y$ implies that $X^s\stprec Y^s$. But this is
true in a stronger stochastic order known as the \emph{likelihood ratio order}.
For integer-valued random variables $X$ and $Y$, we say that $X\lrprec Y$
if $\P(Y=k)/\P(X=k)$ is increasing over the union of the supports of $X$ and $Y$,
interpreting this quantity as $\infty$ when $\P(Y=k)>0$ and $\P(X=k)=0$.
It is not hard to show that $X\lrprec Y$ implies $X\stprec Y$ \cite[Theorem~1.C.1]{SS}.

\begin{proposition}\thlabel{prop:bias.order}
  Let $X$ and $Y$ take values in the nonnegative integers.
  If $X\lrprec Y$, then $X^s\lrprec Y^s$.
\end{proposition}
\begin{proof}
  By definition of the size-bias transform,
  \begin{align*}
    \frac{\P(Y^s=n)}{\P(X^s=n)} = \frac{(\E X) \P(Y=n)}{(\E Y)\P(X=n)}.
  \end{align*}
  This is increasing in $n$ since $X\lrprec Y$.
\end{proof}

\subsection{Log-concave random variables}\label{sec:log.concavity}
  
Let $X$ be a random variable taking values in the nonnegative integers, and let $P_n=\P(X=n)$.
We say that $X$ and its distribution are \emph{log-concave} if 
\begin{enumerate}[(i)]
  \item $P_n^2\geq P_{n-1}P_{n+1}$ for all $n\geq 1$; and
  \item the sequence $P_0,P_1,\ldots$ has no internal zeros (i.e., if $P_i,P_k>0$ for $i<k$,
    then $P_j>0$ for all $i\leq j\leq k$).
\end{enumerate}
Our need for log-concave random variables boils down to the following fact, which will be used
in combination with \thref{prop:bias.order} at a key moment.
\begin{lemma}\thlabel{lem:lr.logconcave}
  Let $X$ be a random variable taking values in the nonnegative integers.
  Then $X\lrprec X+1$ if and only if $X$ is log-concave.
\end{lemma}
\begin{proof}
  Let $P_n=\P(X=n)$ as above. If $(P_n)_{n \geq 1}$ has no internal zeros
  and $N$ is the highest value such that $P_N>0$ (allowing $N=\infty$),
  then the statements
  \begin{align*}
    P_{n-1}P_{n+1}&\leq P_n^2 \quad\text{for all $n\geq 1$}\\\intertext{and}
    \frac{P_{n-1}}{P_n}&\leq \frac{P_n}{P_{n+1}} \quad \text{for all $1\leq n\leq N$}
  \end{align*}
  are equivalent. The first of these statements is the log-concavity of $(P_n)_{n\geq 0}$,
  and the second is the meaning of $X\lrprec X+1$ under the assumption of no internal zeros.
  Thus, the log-concavity of $X$ implies $X\lrprec X+1$. For the other direction, we just
  need to show that $X\lrprec X+1$ implies that $(P_n)_{n\geq 0}$ has no internal zeros.
  Indeed, suppose $P_i,P_j>0$ for $i<j$. 
  Since $P_i/P_{i+1}\leq P_{j-1}/P_j$, we have $P_{i+1}>0$, and then we can proceed
  to show that $P_{i+2}>0$, and so on.
\end{proof}

Next, we state a few technical facts about log-concave random variables, which come from translating
combinatorial results into probabilistic ones. The first result we give is well known;
see \cite[Proposition~2]{Stanley} for the most standard proof,
or see \cite{Liggett} for a completely elementary proof.
\begin{proposition}\thlabel{prop:lc.convolution}
  If $X$ and $Y$ are log-concave and independent, then so is $X+Y$.
\end{proposition}

Next, we show that log-concavity is preserved under thinning, which follows quickly
from a combinatorial result of Brenti's \cite{Brenti}. We suspect that this has been used before,
but we could not find it stated anywhere.
\begin{proposition}\thlabel{prop:lc.thinning}
  Let $X$ be log-concave. Then $\Bin(X,p)$ is also log-concave for any $0<p<1$.
\end{proposition}
\begin{proof}
  First, we prove this under the assumption that $X$ has finite support.
  Let $P_n=\P(X=n)$ as before. 
  We call a sequence $x_0,x_1\ldots$ log-concave if it satisfies the same conditions
  as $P_0,P_1,\ldots$, i.e., no internal zeros and $x_n^2\geq x_{n-1}x_{n+1}$ for $n\geq 1$.
  From the log-concavity of $(P_n)_{n\geq 0}$, it follows by direct calculation
  that the sequence $\{(1-p)^nP_n\}_{n\geq 0}$ is also log-concave. Now, let
  \begin{align*}
    Q_k = \sum_{n= k}^{\infty}\binom{n}{k} (1-p)^n P_n.
  \end{align*}
  By \cite[Theorem~2.5.3]{Brenti}, the sequence $(Q_k)_{k\geq 0}$ is log-concave (this is where we use
  that $X$ has finite support).
  This implies that the sequence
  \begin{align*}
    \biggl\{ \biggl(\frac{p}{1-p}\biggr)^k Q_k\biggr\}_{k\geq 0}
  \end{align*}
  is log-concave, by direct calculation. And now we have proven that $\Bin(X,p)$
  is log-concave, since
  \begin{align*}
    \P\bigl( \Bin(X,p)=k \bigr) &= \sum_{n=k}^{\infty} \binom{n}{k} p^k (1-p)^{n-k}\P(X=n) = 
    \biggl(\frac{p}{1-p}\biggr)^k Q_k.
  \end{align*}

  Finally, we remove the condition that $X$ has finite support with a limit argument.
  Let $X_n$ be distributed as $X$ conditioned on $X\leq n$. Then $X_n\to X$ in distribution.
  Since $X_n$ is log-concave and has finite support, $\Bin(X_n,p)$ is
  log-concave. It is then straightforward to see that $\Bin(X_n,p)$ converges in
  distribution to $\Bin(X,p)$ and that a weak limit of log-concave random variables
  is log-concave.
\end{proof}

\subsection{Upper bounds on root visits}\label{sec:tree.upper}

As we mentioned at the end of Section~\ref{sec:rde}, 
we will define a sequence $(U_n)_{n\geq 0}$ that serves as an upper bound
for $(W_n)_{n\geq 0}$.
Let $U_0=1$ a.s. Then inductively define the sequence by
\begin{align*}
  U_{n+1} \eqd \Aa U_n^>,
\end{align*}
where we use the notation $X^>$ as in \thref{lem:bias.nonzero}.
The overarching goal of the section is to show that $U_n$ admits a size-bias coupling
bounded for the lower tail, which we do by induction. 
\commTJ{As we described following \thref{lem:A.growth}, we can use a concentration bound
for the lower tail of $U_n$ to bound the growth of $\E U_n$, which will allow us to prove
$\E U_n=O(\log n)$. Finally, the same estimate holds for $W_n$ since it is bounded by $U_n$.}
We start by proving that $U_n$ is in fact an upper bound for $W_n$.

\begin{lemma}\thlabel{lem:VU.dominance}
  For all $n\geq 0$,
  \begin{align*}
    W_n\stprec U_n.
  \end{align*}
\end{lemma}
\begin{proof}
  We prove this by induction, starting with $0=W_0\stprec U_0=1$.
  Now suppose that $W_n\stprec U_n$.
  Since $U_n\stprec U_n^>$, we have $W_n\stprec U_n^>$, and applying $\Aa$ to
  both sides of this inequality yields $W_{n+1}\stprec U_{n+1}$ by \commTJ{\thref{lem:AW,lem:A.monotone}}.
\end{proof}

The point of working with $(U_n)$ rather $(W_n)$ is that we will be able to show
that the random variables $(U_n)$ are log-concave. First we give a technical lemma in
this direction.

\begin{lemma}\thlabel{lem:log.concave.tech}
  Let $Y_1,\ldots,Y_d$ be i.i.d.\ random variables taking values $\pm 1$ with $\P(Y_i=1)=p$.
  For $d\ge 2$ and $p\geq 4/9$, the random variable
  \begin{align}\label{eq:log.concave.tech}
    \Bin\Biggl(\sum_{i=1}^d(Y_i+1),\  d^{-1} \Biggr)
  \end{align}
  is log-concave.
\end{lemma}
\begin{proof}
  First, consider the case $d\geq 3$. We claim that for $p\geq 3/7$, the distribution
  $\Bin(Y_i+1,1/3)$ is log-concave. Indeed, let $X$ have this distribution.
  We only need to check that $\P(X=1)^2\geq\P(X=0)\P(X=2)$. A computation shows that
  \begin{align*}
    \P(X=1)^2-\P(X=0)\P(X=2) = \frac{(7p-3)p}{27},
  \end{align*}
  confirming the claim.
  For $x\leq 1/3$, the distribution $\Bin(Y_i+1,x)$ is a thinning of $\Bin(Y_i+1,1/3)$
  and hence is log-concave by \thref{prop:lc.thinning}. By viewing \eqref{eq:log.concave.tech}
  as
  \begin{align*}
    \sum_{i=1}^d\Bin\bigl(Y_i+1,\ d^{-1}\bigr),
  \end{align*}
  we see that it is a sum of independent log-concave random variables for any $d\geq 3$
  and by \thref{prop:lc.convolution} is log-concave.
  
  Now we turn to the case $d=2$. 
  Let $X$ be distributed as \eqref{eq:log.concave.tech}, and
  let $P_k=\P(X=k)$ for $k=0,\ldots,4$.
  We then compute directly
  \begin{align*}
    P_1^2-P_0P_2 &= \frac{(9p-4)(3p-4)^2p}{128},\\
    P_2^2-P_1P_3 &= \frac{\Bigl(13\bigl(p-\frac{12}{13}\bigr)^2+\frac{64}{13}\Bigr)p^2}{64},\\
    P_3^2-P_2P_4 &= \frac{(9p-4)p^2}{128},
  \end{align*}
  all of which are positive for $p\geq 4/9$.
\end{proof}

\begin{proposition}\thlabel{lem:log.concave}
  If $p\geq 4/9$, then $U_n$ is log-concave for all $n\geq 0$.
\end{proposition}
\begin{proof}
  We prove this by induction. 
  Trivially, $U_0$ is log-concave.
  Now, assume that $U_n$ is log-concave.
  This implies that $U_n^>$ is log-concave, since the sequence $(\P(U_n^>=k))_{k\geq 0}$
  is a rescaled version of the sequence $(\P(U_n=k))_{k\geq 0}$ with the $k=0$ term
  set to zero. 
  Let $X_1,\ldots,X_d$ be i.i.d.\ copies of $U_n^>$, and let $Y_1,\ldots,Y_d$
  be i.i.d.\ random variables taking values $\pm 1$ with $\P(Y_i=1)=p$.
  Looking back at the definition of $\Aa$ in \eqref{eq:Adef}, the distribution of $U_{n+1}$ is given by
  \begin{align}
    \Aa U_n^> &= \Bin\Biggl(\sum_{i=1}^d (X_i+Y_i)^+,\  d^{-1} \Biggr)=\Bin\Biggl(\sum_{i=1}^d (X_i+Y_i),\  d^{-1} \Biggr)\label{eq:AUn1}
  \end{align}
  \commTJ{with the replacement of $(X_i+Y_i)^+$ by $X_i+Y_i$ justified by the strict positivity of $X_i$.
  Now we take advantage of this positivity to break $X_i+Y_i$ into nonnegative
  terms $X_i-1$ and $Y_i+1$, giving us}
  \begin{align}
     \Aa U_n^>&= \Bin\Biggl(\sum_{i=1}^d (X_i-1)+\sum_{i=1}^d(Y_i+1),\  d^{-1} \Biggr)\\
             &= \Bin\Biggl(\sum_{i=1}^d (X_i-1),\ d^{-1}\Biggr)+
                \Bin\Biggl(\sum_{i=1}^d(Y_i+1),\  d^{-1} \Biggr).\label{eq:AUn2}
  \end{align}
  Since $X_i$ is log-concave, so is $X_i-1$. Thus the first term on the right-hand side 
  in \eqref{eq:AUn2}
  is a thinned sum of log-concave random variables and hence is log-concave by
  \thref{prop:lc.convolution,prop:lc.thinning}. 
  The second term is log-concave by \thref{lem:log.concave.tech}.
  A final application of \thref{prop:lc.convolution} shows that \eqref{eq:AUn2} is 
  a log-concave distribution, completing the induction.
\end{proof}

Now, we start seting up an induction to show that 
$U_n$ admits a size-bias coupling that is $(1,q)$-bounded for the lower tail for $q>0$.
\begin{lemma}\thlabel{lem:spine.induction}
  Let $X_1,\ldots,X_d$ be distributed as $U_n^>$,
  and let $Y_1,\ldots,Y_d$ be i.i.d.\ random variables taking values $\pm 1$ with $\P(Y_i=1)=p\geq 4/9$, 
  and let all be independent of each other.
  Let $U_n^s$ be independent of $X_2,\ldots,X_d$ as well.
  For all $n\geq 0$,
  \begin{align*}
    U_{n+1}^s \stprec 1 + \Bin\Bigl(U_n^s + (Y_1+1)^s - 1,\ d^{-1}\Bigr)
         + \Bin\Biggl(\sum_{i=2}^d(X_i+Y_i),\ d^{-1}\Biggr)
  \end{align*}
\end{lemma}
\begin{proof}
  As in \eqref{eq:AUn1}, 
  \begin{align*}
    U_{n+1}\eqd \sum_{i=1}^d\Bin\bigl( X_i+Y_i,\;  d^{-1} \bigr).
  \end{align*}
  By \thref{lem:bias.sum}, we obtain the size-bias transform of $U_{n+1}$ by
  choosing a term in this sum at random and biasing it. Since all terms are identically
  distributed, we can just bias the first term. Applying \thref{lem:bias.thinning} to this term,
  we obtain
  \begin{align}
    U_{n+1}\eqd 1 + \Bin\bigl((X_1+Y_1)^s - 1,\; d^{-1}\bigr)
      + \sum_{i=2}^d \Bin\bigl( X_i+Y_i,\; d^{-1} \bigr).\label{eq:Un1.expr}
  \end{align}
  Now, we stochastically bound $(X_1+Y_1)^s$.
  We first rewrite it as $\bigl((X_1-1) + (Y_1+1)\bigr)^s$, noting that $X_1-1$ and $Y_1+1$
  are both nonnegative. By \thref{lem:bias.sum},
  \begin{align}
    (X_1+Y_1)^s \eqd
      \begin{cases}
        (X_1-1)^s + (Y_1+1) & \text{with probability $a$,}\\
        (X_1-1) + (Y_1+1)^s & \text{with probability $1-a$,}
      \end{cases}\label{eq:XY.mixture}
  \end{align}
  for some $a$ that we will avoid calculating.  
  
  Now, we show that both parts of this mixture are stochastically dominated
  by $U_n^s + (Y_1+1)^s$. First observe that $U_n$ is log-concave by \thref{lem:log.concave}, and therefore
  so is $X_1\eqd U_n^>$. Thus $X_1-1$ is log-concave, and $X_1-1\lrprec X_1$
  by \thref{lem:lr.logconcave}.
  Applying \thref{prop:bias.order} and the fact that dominance in the lr-order
  implies dominance in the usual order, we see that $(X_1-1)^s\stprec X_1^s\eqd U_n^{>s}$. Finally,
  $U_n^{>s}\eqd U_n^s$ by \thref{lem:bias.nonzero}, establishing that $(X_1-1)^s\preceq U_n^s$.
  Next, we have $Y_1+1\stprec (Y_1+1)^s$, a fact about size-biasing (it holds because
  $X\lrprec X^s$ in general by direct calculation).
  This completes the proof that the first part of the mixture is dominated by
  $U_n^s + (Y_1+1)^s$. To show this for the second part, we simply note that $X_1-1\preceq (X_1-1)^s\preceq U_n^s$.
  
  Since $U_n^s+(Y_1+1)^s$ stochastically dominates
  both parts of the mixture in \eqref{eq:XY.mixture}, it stochastically
  dominates $(X_1+Y_1)^s$ as well.
  Applying this to \eqref{eq:Un1.expr} completes the proof.
\end{proof}

Now we show the existence of our size-bias coupling for $U_n$. The idea
is to iterate the stochastic relation proven in \thref{lem:spine.induction}
to obtain a stochastic relation between $U_n$ and $U_n^s$ that implies the
existence of the coupling.
\begin{proposition}\thlabel{prop:sb.coupling}
  Let $p\geq 4/9$ and $q=\prod_{i=1}^{\infty} \bigl(1 - 2^{-i}\bigr)^2\approx .083$. For all $n\geq 0$,
  the random variable $U_n$ admits a size-bias coupling that is $(1,q)$-bounded
  for the lower tail.
\end{proposition}
\begin{proof}
  The bulk of the proof is to show that for all $n\geq 1$,
  \begin{align}\label{eq:stochastic.sb.coupling}
    U_n^s \stprec 1 + \sum_{i=1}^{n}\Bin\bigl(2, d^{-i}\bigr) + U_n,
  \end{align}
  which we prove by induction.
  Let $X_1,\ldots,X_d$ be distributed as $U_n^>$,
  and let $Y_1,\ldots,Y_d$ be i.i.d.\ random variables taking values $\pm 1$ with $\P(Y_i=1)=p$, 
  all independent of each other. For the base case, set $n=1$. Then by \thref{lem:spine.induction},
  \begin{align*}
    U_1^s &\stprec 1 + \Bin\Bigl(U_0^s + (Y_1+1)^s-1,\ d^{-1}\Bigr)
        + \Bin\Biggl(\sum_{i=2}^d(X_i+Y_i),\ d^{-1}\Biggr)\\
        &\stprec 1 + \Bin\bigl(2,\ d^{-1}\bigr)
        + \Bin\Biggl(\sum_{i=2}^d(X_i+Y_i),\ d^{-1}\Biggr)\\
        &\stprec 1 + \Bin\bigl(2,\ d^{-1}\bigr)
        + \Bin\Biggl(\sum_{i=1}^d(X_i+Y_i),\ d^{-1}\Biggr)\eqd 1 + \Bin\bigl(2,d^{-1}\bigr) + U_1.
  \end{align*}
  From the first to the second line, we have applied the inequality $U_0^s+(Y_1+1)^s-1\leq 2$, since $U_0^s=1$ and $(Y_1+1)^s\leq 2$. The last equality is due to the definition of $U_{1}$.

  Now, we assume that \eqref{eq:stochastic.sb.coupling} holds for $n$, and we show it holds for $n+1$.
  Apply \thref{lem:spine.induction} together with the inductive hypothesis to obtain
  \begin{align*}
    U_{n+1}^s &\stprec 1 + \Bin\Biggl(1 + \sum_{i=1}^n\Bin\bigl(2, d^{-i}\bigr) + U_n + (Y_1+1)^s - 1,\ d^{-1}\Biggr)  \\
    		& \qquad \qquad \qquad \qquad \qquad \qquad\qquad\quad  + \Bin\Biggl(\sum_{i=2}^d(X_i+Y_i),\ d^{-1}\Biggr).
  \end{align*}
  Again, we have $(Y_1+1)^s\leq 2$.
  We can also apply the bound $U_n\stprec U_n^>\eqd X_1$. This yields
    \begin{align*}
    U_{n+1}^s &\stprec 1 + \Bin\Biggl(\sum_{i=1}^{n}\Bin\bigl(2, d^{-i}\bigr) + X_1 + 2,\ d^{-1}\Biggr)
          + \Bin\Biggl(\sum_{i=2}^d(X_i+Y_i),\ d^{-1}\Biggr)\\
         &\eqd 1 + \sum_{i=1}^{n+1}\Bin\bigl(2,d^{-i}\bigr)
           + \Bin\bigl(X_1,d^{-1}\bigr)+ \Bin\Biggl(\sum_{i=2}^d(X_i+Y_i),\ d^{-1}\Biggr)\\
         &\stprec  1 + \sum_{i=1}^{n+1}\Bin\bigl(2,d^{-i}\bigr)+ U_{n+1},
  \end{align*}
  where the last domination is due to the definition of $U_{n+1}$ (see \eqref{eq:AUn1}). Advancing the induction completes the proof of \eqref{eq:stochastic.sb.coupling}.
  
  The stochastic inequality \eqref{eq:stochastic.sb.coupling}
  implies the existence of a coupling of $U_n$ and $U_n^s$
  under which the inequality
  holds almost surely. All that remains is to show that this coupling is $(1,q)$-bounded for the lower
  tail. For this, we observe that
  \begin{align*}
    \P\Biggl( \sum_{i=1}^{n}\Bin\bigl(2, d^{-i}\bigr) = 0 \Biggr) 
      \geq \P\Biggl( \sum_{i=1}^{\infty}\Bin\bigl(2, d^{-i}\bigr) = 0 \Biggr)
       = \prod_{i=1}^{\infty} \bigl(1-d^{-i}\bigr)^2\geq q.
  \end{align*}
  Hence, under this coupling,
  \begin{align*}
    \P\bigl(U_n^{s} \leq U_n + 1\mid U_n\bigr) &\geq q.
  \end{align*}
\end{proof}

Here in \thref{lem:spine.induction} and \thref{prop:sb.coupling}, we finally
see how size-biasing creates a spine in $\BT$ along which there are extra visits
to the root, as we described in Section~\ref{sec:proof.overview}.
The $1 + \sum_{i=1}^{n}\Bin\bigl(2, d^{-i}\bigr)$ term in
\eqref{eq:stochastic.sb.coupling} is an overestimate of these extra returns.
Heuristically, each $\Bin\bigl(2, d^{-i}\bigr)$ term represents two extra particles
at level~$i$ on the spine. Because a particle has a $d^{-1}$ chance of moving toward
the root at each step, the chance of a particle at level~$i$ reaching the root
is only $d^{-i}$, and we get only $O(1)$ expected extra visits to the root.

\begin{remark}
  The random variables $(U_n)$ were introduced because they could be proven to be log-concave.
  But computer investigations suggest that the random variables $(W_n)$ themselves
  are log-concave, at least in the $p=1/2$ case. If we could prove this, it would allow
  us to improve \thref{thm:tree} to a result on density. The proof would go as follows.
  First, $W_n$ would admit a size-bias coupling as in the previous
  proposition. Together with \thref{prop:tree.discrete.bound}\ref{i:tdb.critical}, this would
  show that $\P(W_n=0)=O(1/n)$, from which it would follow that $\P(V_t=0)=O(1/t)$.
  This implies that a $B$-particle at the root would survive for time~$t$ with probability $O(1/t)$,
  proving that the density of $B$-particles decays at rate $O(1/t)$. Since the density
  of $A$- and $B$-particles is the same, this would show that $\rho_t=O(1/t)$.
\end{remark}

The last ingredient is a technical lemma that bounds the growth of a sequence satisfying a certain recursive bound.

\begin{lemma}\thlabel{lem:recursion.log.bound}
	Suppose a sequence $\mu_n$ satisfies 
  \begin{align}\label{eq:rlb.hypo}
    \mu_{n+1} \leq \frac{\mu_n}{1 - e^{-q \mu_n}}
  \end{align}
  for some $q>0$. Then there exists $C>0$ such that $\mu_n \leq C \log n$ for all $n \geq 2$. 
\end{lemma}

\begin{proof}
We proceed by induction. Choose the constant $C\geq 2/q$ large enough
  that $\mu_2\leq C\log 2$, establishing the base case of the induction.
  Now, we assume $\mu_n\leq C\log n$ and advance the induction.
  Since $x\mapsto x/(1-e^{-qx})$ is increasing, the inductive hypothesis together with
  \eqref{eq:rlb.hypo} and our assumption $C\geq 2/q$ yields
  \begin{align}\label{eq:ub.sufficient2}
    \mu_{n+1} \leq \frac{C\log n}{1-n^{-2}}.
  \end{align}
  Hence in order to show $\mu_{n+1}\le C\log (n+1)$, it suffices to show that
  \begin{align}\label{eq:ub.sufficient}
    \frac{\log n}{1-n^{-2}} \leq \log(n+1).
  \end{align}
  To prove this, start with the inequality $\log n \leq n-1$ for $n\geq 1$.
  Dividing both sides by $n^2-1$ yields
  \begin{align*}
    \frac{\log n}{n^2-1} \leq \frac{1}{n+1}.
  \end{align*}
  Rewriting the left-hand side,
  \begin{align*}
    \frac{\log n}{1-n^{-2}} - \log n \leq \frac{1}{n+1}.
  \end{align*}
 Then adding $\log n$ to both sides and applying the inequality $(\log n) + 1/(n+1)\leq \log(n+1)$ shows \eqref{eq:ub.sufficient}. This completes the induction and proves that $\mu_n\leq C\log n$ for all $n\geq 2$.
\end{proof}

\begin{proof}[Proof of upper bounds in Proposition \ref{prop:tree.discrete.bound}]
  We will bound $\E U_n$, which bounds $\E W_n$
  by \thref{lem:VU.dominance}. Let $\mu_n= \E U_n$.
  We start by applying
  \thref{prop:sb.coupling} and \thref{thm:concentration}\ref{i:concentration.lower} to deduce that
  \begin{align*}
    \P(U_n=0)\leq e^{-q\mu_n\commHL{/2}},
  \end{align*}
  where $q$ is the constant from \thref{prop:sb.coupling}.
  Thus,
  \begin{align}\label{eq:Un>}
    \E U_n^> = \frac{\mu_n}{1 - \P(U_n=0)}\leq \frac{\mu_n}{1-e^{-q\mu_n\commHL{/2}}}.
  \end{align}

  Now consider the critical case $p=1/2$.
  If $X\geq 1$ a.s., then \commHL{$\E[\Aa X]=\E X$} by \thref{lem:A.growth}. 
  Thus $\mu_{n+1}=\commHL{\E[\Aa U_n^{>}]}=\E U_n^>$, which with \eqref{eq:Un>} shows that
  \begin{align}\label{eq:mu.induction}
    \mu_{n+1}\leq \frac{\mu_n}{1-e^{-q\mu_n\commHL{/2}}}.
  \end{align}    
  It follows from \thref{lem:recursion.log.bound} that $\mu_n \leq C \log n$ for some $C>0$ 
  and all $n \geq 2$. 
  By \thref{lem:VU.dominance}, we have $\E W_n\leq \mu_n$.
  This completes the proof of the upper bound in part~\ref{i:tdb.critical}, the
  critical case.
  
  We now prove the upper bound on $\lim_{n\rightarrow \infty}\E W_n$ in the $p<1/2$ case.
  Since $\E W_n$ is an increasing sequence, we need only show that
  $\E W_n$ is bounded by $C\log(1/\epsilon)$ where $\epsilon=(1/2)-p$.
  By \thref{lem:A.growth} and \eqref{eq:Un>},
  \begin{align*}
    \mu_{n+1}=\commHL{\E[\Aa U_n^{>}]}=\E U_n^> -2\epsilon\leq \frac{\mu_n}{1-e^{-q\mu_n\commHL{/2}}}-2\epsilon.
  \end{align*} 
  Let $\varphi(x)= x/(1-e^{-qx\commHL{/2}})-2\epsilon$ for $x>0$,
  and set $\varphi(0)=\lim_{x\to 0}\varphi(x)=\commHL{\frac{1}{2q}}-2\epsilon$.
  Some calculus shows that
  the function $x\mapsto\varphi(x)-x$ is positive at $x=0$, is strictly decreasing on $[0,\infty)$,
  and approaches a limit of $-2\epsilon$ as $x\to\infty$. Hence,
  $\varphi$ has a unique fixed point $x_0$.
  Since $\varphi$ is strictly increasing, if $x<x_0$, then $\varphi(x)<\varphi(x_0)=x_0$.
  It follows from $q,\epsilon<1/2$ that $\varphi(1)-1>0$; hence $x_0>1$.
  Thus $1=\mu_0<x_0$, and hence $\mu_n<x_0$ for all $n$.
  By \thref{lem:VU.dominance}, we have $\E W_n\leq x_0$ for all $n$.
  All that remains is to
  estimate $x_0$.
  For $x=\frac{2}{q}\log(1/\epsilon)$,
  \begin{align*}
    \varphi(x)-x = \frac{2\epsilon^2}{q(1-\commHL{\epsilon})}\log(1/\epsilon)-2\epsilon,
  \end{align*}
  which is negative when $\epsilon$ is small. Hence $x_0<\frac{2}{q}\log(1/\epsilon)$
  for all sufficiently small $\epsilon$.
  \end{proof}

\subsection{Lower bounds on root visits}\label{tree.lower}
The lower bounds in \thref{prop:tree.discrete.bound} are much simpler than the upper bounds.
By \thref{lem:A.growth}, a lower bound on $\E W_n$ follows from anticoncentration
estimates (i.e., lower bounds on $\P(W_n=0)$). The idea of the proof is that even if we assume
that $W_n$ is as concentrated as possible---that its distribution is a point mass---then
the distribution of $W_{n+1}\eqd\Aa W_n$ is still nonconcentrated enough to yield the correct
lower bound.

\begin{lemma}\thlabel{lem:lower.nonconcentration}
  For all $n\geq 0$ and $p\leq 1/2$,
  \begin{align*}
    \P(W_{n+1}=0) \geq 4^{-\E W_n - 1}.
  \end{align*}
\end{lemma}
\begin{proof}
  Define $X=\max(W_n,1)$.
  \commTJ{By \thref{lem:A.monotone}} we have $\Aa X\stsucc \Aa W_n$. By \thref{lem:AW},
  \begin{align*}
    \P( W_{n+1}=0)\geq \P(\Aa X = 0).
  \end{align*}
  Let $X_1,\ldots,X_d$ be i.i.d.\ copies of $X$, and let $Y_1,\ldots,Y_d$ be i.i.d.\ 
  with $\P(Y_i=1)=p$ and $\P(Y_i=-1)=1-p$.
  By the definition of $\Aa$ given in \eqref{eq:Adef} and the
  fact that $X_i\geq 1$,
  \begin{align*}
    \P(\Aa X = 0) = \E\Biggl[ \bigl(1 - 1/d\bigr)^{\sum_{i=1}^d(X_i+Y_i)}\Biggr]
      &= \prod_{i=1}^d \E\Biggl[ \bigl(1 - 1/d\bigr)^{X_i} \Biggr]\,
        \E\Biggl[ \bigl(1 - 1/d\bigr)^{Y_i} \Biggr]\\
        &\geq \bigl(1 - 1/d\bigr)^{d\E X}
        \geq 4^{-\E W_n-1}.
  \end{align*}
  From the first to the second line, we have applied Jensen's inequality to both expectations
  and then observed that $(1-1/d)^{\E Y_i}\geq 1$, since $\E Y_i\leq 0$ from our assumption $p\leq 1/2$.
  The last inequality holds because $\E X\leq \E W_n+1$,
  and $(1-1/d)^d\geq 1/4$ for $d\geq 2$.
\end{proof}

\begin{proof}[Proof of lower bounds in \thref{prop:tree.discrete.bound}]
  Let $\mu_n=\E W_n$.
  It follows from Lemmas~\ref{lem:lower.nonconcentration} and \ref{lem:A.growth} that
  for all $n\geq 0$ and $\epsilon=1/2-p\geq 0$,
  \begin{align}\label{eq:mu.DE}
    \mu_{n+2} \geq \mu_n -2\epsilon + \bigl(\epsilon+1/2\bigr)4^{-\mu_n-1}.
  \end{align}
  Since $x\mapsto x -2\epsilon + \bigl(\epsilon+1/2\bigr)4^{-x-1}$ is increasing for all $x\geq 0$,
  \begin{align}\label{eq:mutonicity}
    \mu_n\geq x \quad \implies\quad  \mu_{n+2}\geq x-2\epsilon+\bigl(\epsilon+1/2\bigr)4^{-x-1}.
  \end{align}
  
  Now, consider the $p=1/2$ case. Choose $c\leq 1/16$, observing that it then holds for all $n\geq 1$
  that
  \begin{align}\label{eq:2c/n}
    2c\leq n^{1-c\log 4}/8.
  \end{align}
  Also take $c$ small enough that  
  $\mu_2\geq c\log 2$.
  Using this statement and $\mu_1\geq 0$ as base cases, we prove
  \begin{align}\label{eq:mlb.hyp}
    \mu_n\geq c\log n
  \end{align}
  by carrying out one induction for odd $n$ and one for even $n$.
  Assuming \eqref{eq:mlb.hyp}, we apply
  \eqref{eq:mutonicity} to get
  \begin{align*}
    \mu_{n+2} \geq c\log n + n^{-c\log 4}/8\geq c(\log n + 2/n),
  \end{align*}
  applying \eqref{eq:2c/n} for the last inequality.
  Since $\log x$ is concave, we have $\log n+2/n\geq \log(n+2)$,
  advancing the induction. This completes the proof for the $p=1/2$ case.
  
  For $p<1/2$, we already know from the upper bounds of \thref{prop:tree.discrete.bound}
  that $\mu:=\lim_n\mu_n$ exists, since $\mu_n$ is increasing. Let
  \begin{align*}
    x_0 = -1+\frac{1}{\log 4}\log\Biggl(\frac{\epsilon+1/2}{2\epsilon}\Biggr),
  \end{align*}
  which is the solution for $x$ to the equation $x=x-2\epsilon+\bigl(\epsilon+1/2\bigr)4^{-x-1}$.
  We claim that $\mu\geq x_0$. Indeed, if not, then by taking $\mu_n$ sufficiently close to
  $\mu$, we could conclude from \eqref{eq:mutonicity} that $\mu_{n+2}>\mu$, which is a contradiction. 
\end{proof}


\begin{appendix}

\section{Random walk estimates}

We start with some standard facts about the maximum displacement of random walk.
The first is a moderate deviations estimate for a (possibly) biased continuous-time
random walk.

  \begin{lemma}\thlabel{lem:SRW}
    Let $S_t$ be a rate~$1$ continuous-time nearest-neighbor random walk on $\ZZ$ started at the origin with expected increment $\eps\in [0,1]$. Let $M_t = \sup_{0\leq s \leq t} S_t$. For all $0\leq x\leq 2t$,
    \begin{align*}
      \P(M_t\geq x+2\epsilon t) \leq e^{-x^2/4t}.
    \end{align*}
  \end{lemma}

   \begin{proof}
   Applying Doob's martingale inequality to $S_t-\epsilon t$,
    \begin{align}\label{eq:doob}
      \P(M_t\geq x + 2\epsilon t) &\leq \P\Biggl(\sup_{0\leq s\leq t} (S_s-\epsilon s)\geq x+\epsilon t\Biggr)
        \leq \frac{\E\exp\bigl(\lambda (S_t-\epsilon t)\bigr)}{\exp\bigl(\lambda(x+\epsilon t)\bigr)}
    \end{align}
    for all $\lambda$.
   	The moment generating function of a single random step is
    $\cosh\lambda + \epsilon\sinh\lambda$.
    Let $N \eqd \Poi(t)$ be the number of steps taken by the random walk up to time~$t$. 
    Recalling that $\E[z^N] = \exp( t(z-1))$,
   	\begin{align}
   	  \E \exp(\lambda S_t) = \E (\cosh \lambda+\epsilon\sinh\lambda) ^N 
        = \exp\bigl(t (\cosh \lambda + \epsilon \sinh\lambda- 1)\bigr) \label{eq:mgf}.
   	\end{align}
   Applying \eqref{eq:mgf} to \eqref{eq:doob} with $\lambda = x/t$ gives
   	\begin{align}
   	\P\left(M_t\geq x+2\epsilon t\right) \leq \exp\Biggl( t \biggl(\cosh\frac{x}{t}+\epsilon\sinh\frac{x}{t} -1\biggr) - 2\epsilon x - \frac{x^2}{t}\Biggr).\label{eq:exponent}
   	\end{align}
    For $0\leq \theta\leq 2$, one can confirm using Taylor series that
    $\cosh\theta-1\leq \frac34 \theta^2$ and $\sinh\theta\leq 2\theta$.
    Applying this to  \eqref{eq:exponent} together with  our assumption $0\leq x\leq 2t$
    gives the claimed result.
\end{proof}

\begin{lemma} \thlabel{lem:poisson_tail}
	Let $S_t$ and $M_t$ be as in \thref{lem:SRW}. For any $k \geq 1$ it holds that 
	$$\P(M_t \geq t+k) \leq \exp\left( \f{ - k^2}{2(t+k)} \right).$$
\end{lemma}

\begin{proof}
The displacement of the random walk is bounded by the total number of steps it takes in time~$t$,
a Poisson random variable with mean $t$. The estimate is then a standard tail bound for the Poisson	(see \cite[Chapter 2]{boucheron2013concentration}).
\end{proof}

We also need a quick estimate on the local time at the origin of a random walk after time~$t$.
  
\begin{lemma} \thlabel{lem:local_time}
Let $L_t$ be the time spent at the origin after $t$ steps of a simple random walk $(S_s)_{s\leq t}$ started at the origin. It holds that $\E L_t \leq  \sqrt t$.	
\end{lemma}
\begin{proof}
Applying Tanaka's formula for the local time \cite{kallenberg2006foundations} gives $\E L_t = \E|S_t|$. It is a standard and straightforward recursion to show that the expected value of the square of a simple random walk after $n$ steps is $n$. Hence, conditional on $S_t$ taking $N\overset d = \Poi(t)$ steps, we have
$$\E [|S_t| \mid N] \leq \sqrt{\E [ S_t^2 \mid N]} = \sqrt {N}.$$
It follows by taking expectation and applying Jensen's inequality that $$\E |S_t| \leq \E \sqrt N \leq \sqrt { \E N} = \sqrt t.$$ 
Hence $\E L_t \leq \sqrt t$. 
\end{proof}

Our final goal is to prove the following result, which combines the gambler's ruin
computation of the probability of a random walk hitting $b$ before $-a$ with the moderate
deviations bound given in \thref{lem:SRW}:

\begin{proposition}\thlabel{lem:gamblers.ruin}
  Let $T_x$ be the hitting time of $x$ for a continuous-time simple random walk on $\ZZ$
  started at the origin. For any $t>0$ and integers $a>0$ and $3\sqrt{t}\leq x\leq 2t$,
  \begin{align*}
    \P( \text{$T_x < T_{-a}$ and $T_x\leq t$} ) \leq \frac{2a}{a+x} e^{-x^2/12t}.
  \end{align*}
\end{proposition}

Essentially, this result is that the moderate deviations result 
\thref{lem:SRW} still holds (with a worse
constant in the exponent) after conditioning the random walk to hit $x$ before $-a$.
Heuristically, we can see that this should be true by considering the analogous situation
for Brownian motion. Conditioning the random walk to hit $x$ before $-a$ is like conditioning
Brownian motion to stay positive, which makes it a Bessel-3 process, the absolute value of 
a $3$-dimensional Brownian motion. From this explicit representation, it is easy to check that it satisfies
a moderate deviations tail bound.

To prove this result for random walks, we start by
considering two random walks, one unbiased and one with a bias in the positive direction.
We will need a lemma establishing the highly intuitive fact that the biased walk is faster to hit $b$
even after conditioning both walks to hit $b$ before $-a$, for any $a,b>0$.
\begin{lemma}\thlabel{lem:conditioned.dominance}
  Let $T_m$ be the first hitting time of $m$ for a continuous-time simple random walk
  started from the origin, and let $T'_m$ be the hitting time for a random walk
  that jumps to the right with probability $p>1/2$. For any positive integers $a$ and $b$,
  the conditional distribution of $T_b$ given $T_b<T_{-a}$ is stochastically larger
  than that of $T'_b$ given $T'_b<T'_{-a}$.
\end{lemma}

\begin{proof}  
  Let $(S_n)$ be simple random walk conditioned to hit $b$ before
  $-a$, and let $(S'_n)$ be the biased walk under the same conditioning, both in discrete time.
  Take $U$ and $U'$ to be the hitting times of $b$ for these processes.
  It suffices to prove $U'\preceq U$,
  since we can couple the continuous-time walks to follow the paths of the discrete-time
  walks with identical jump times. 
  Let $\Ss_m$ denote the set of nearest-neighbor walks of length $m$ from $0$ to $b$ that
  never hit $-a$, and note that both
  $(S_n)_{n=0}^U$ and $(S'_n)_{n=0}^{U'}$ are supported on the set $\cup_m\Ss_m$.
  Observe that each walk in $\Ss_m$ takes
  $r=(m+b)/2$ jumps to the right and $\ell=(m-b)/2$ jumps to the left. 
  For any $(s_0,\ldots,s_m)\in\Ss_m$,
  \begin{align*}
    \frac{\P\Bigl( (S'_n)_{n=0}^{U'} = (s_n)_{n=0}^m \Bigr)}
           {\P\Bigl( (S_n)_{n=0}^{U} = (s_n)_{n=0}^m \Bigr)}
       &= \frac{p^r(1-p)^\ell/\P(T'_b<T'_{-a})}{2^{-m}/\P(T_b<T_{-a})}=h_0\bigl(2\sqrt{p(1-p)}\bigr)^m,
  \end{align*}
  where
  \begin{align*}
    h_0 &= \frac{p^{b/2}\P(T_b<T_{-a})}{(1-p)^{b/2}\P(T'_b<T'_{-a})}.
  \end{align*}
  Thus,
  \begin{align*}
    \frac{\P(U'=m)}{\P(U=m)} &= \frac{\sum_{(s_n)\in\Ss_m}\P\Bigl( (S'_n)_{n=0}^{U'} = (s_n)_{n=0}^m \Bigr)}{\sum_{(s_n)\in\Ss_m}\P\Bigl( (S_n)_{n=0}^{U} = (s_n)_{n=0}^m \Bigr)}
    =h_0\bigl(2\sqrt{p(1-p)}\bigr)^m.
  \end{align*}
  This is decreasing in $m$, which proves $U'\preceq U$.
\end{proof}

\begin{proof}[Proof of \thref{lem:gamblers.ruin}]
  The gambler's ruin calculation states that $\P(T_x<T_{-a})=a/(a+x)$.
  Thus our goal is to prove that
  \begin{align}\label{eq:gamblers.reduction}
    \P(T_x\leq t \mid T_x<T_{-a}) \leq 2e^{-x^2/12t}.
  \end{align}
  To show this, fix $t$, and let $(S_s)$ be a continuous-time random walk starting
  from 
  \commHL{0}
  with probability
  $(1+t^{-1/2})/2$ of jumping to the right.
  By \thref{lem:conditioned.dominance}, it suffices to bound its probability
  of hitting $x$ in time $t$ given that it hits $x$ before $-a$.
  Since $(S_s)$ must pass through $\ceil{x/2}$ on its way to $x$ and is
  a Markov process even after conditioning, it suffices to prove the bound under
  the assumption that $(S_s)$ starts at $\ceil{x/2}$ rather than the origin.
  Thus, our goal now is to show that
  \begin{align}
    \P(U_x\leq t \mid U_x<U_{-a}) \leq 2e^{-x^2/12t},
  \end{align}
  where $U_y$ is the first hitting time of $y$ for the biased random walk starting
  at $\ceil{x/2}$. By the gambler's ruin calculation for biased random walks,
  \begin{align*}
    \P(U_x<U_{-a}) = \frac{1-\alpha^{\ceil{x/2}+a}}{1-\alpha^{x+a}},
  \end{align*}
  where $\alpha=(1-t^{-1/2})/(1+t^{-1/2})$. Using $\alpha\leq 1-t^{-1/2}$
  and our assumption $x\geq3\sqrt{t}$,
  \begin{align*}
    \P(U_x<U_{-a}) \geq 1-\alpha^{x/2+a}\geq 1 - (1-t^{-1/2})^{3\sqrt{t}/2+a}\geq 1-e^{-3/2}.
  \end{align*}
  Hence, by \thref{lem:SRW},
  \begin{align*}
    \P(U_x\leq t \mid U_x<U_{-a}) &\leq \frac{\P(U_x\leq t)}{\P(U_x<U_{-a})}
      \leq\frac{ \exp\Bigl( -\frac{x-2\sqrt{t}}{4t} \Bigr)}{1-e^{-3/2}}
      \leq 2e^{ -x/12t},
  \end{align*}
  using the bound $x-2\sqrt{t}\geq x/3$ that follows from our assumption $x\geq 3\sqrt{t}$.
  This completes the proof by establishing \eqref{eq:gamblers.reduction}.
\end{proof}

\end{appendix}

\begin{acks}[Acknowledgments]
 We thank Michael Damron for helpful feedback and his assistance with Section \ref{sec:EV_LB}.
\end{acks}

\begin{funding}
Johnson was partially supported by the NSF-DMS Grant \#1811952, Junge by the NSF-DMS Grant \#185551, and Lyu by DMS-2206296 and DMS-2010035. Lyu and Sivakoff were partially supported by the NSF grant CCF--1740761.
\end{funding}



\bibliographystyle{imsart-nameyear.bst} 

\bibliography{DLAS}

\end{document}